\documentclass[11pt]{amsart}
\usepackage{amssymb,amsmath,amsthm,latexsym}
\usepackage[tmargin=1in,bmargin=1in,lmargin=1in,rmargin=1in]{geometry}
\usepackage{mathrsfs,cancel,hyperref,verbatim}
\usepackage{enumitem}
\usepackage{stmaryrd} 
\usepackage{pdfpages}
\usepackage{subcaption}
\usepackage{tikz}
\usepackage{tikz-cd}
\usetikzlibrary{angles,quotes,intersections}

\usetikzlibrary{decorations.markings} 

\tikzset{
    midarrow/.style={
        decoration={markings, mark=at position 0.5 with {\arrow{>}}},
        postaction={decorate}
    }
}

\usepackage[utf8]{inputenc}

\author{Hyun Chul Jang}
\address{Department of Mathematics, Texas State University, San Marcos, TX 78666}
\email{hcjang@txstate.edu}
\title{Entropy stability for products of negatively curved symmetric spaces}
\date{\today}

\begin{document}

\newtheorem{thm}{Theorem}
\newcounter{alphthm}
\renewcommand{\thealphthm}{\Alph{alphthm}}
\newtheorem{alphathm}[alphthm]{Theorem}

\newtheorem{theorem}{Theorem}[section] 
\newtheorem{claim}[theorem]{Claim}
\newtheorem{lemma}[theorem]{Lemma}
\newtheorem{proposition}[theorem]{Proposition}
\newtheorem{corollary}[theorem]{Corollary}
\theoremstyle{definition}
\newtheorem{definition}[theorem]{Definition}
\newtheorem{notation}[theorem]{Notation}
\newtheorem{defn}{Definition}
\theoremstyle{remark}
\newtheorem{remark}[theorem]{Remark}

\begin{abstract}
	Let $(M,g_0)$ be a closed oriented $n$-manifold that is locally isometric to a product $(X^{n_1}_1,g_1)\times\cdots (X_k^{n_k},g_k)$, where each $n_i\ge 3,$ and each factor $(X_i^{n_i},g_i)$ is a negatively curved symmetric space. We study the stability of minimal entropy rigidity for such manifolds. Specifically, we consider whether an entropy-minimizing sequence $(M,g_i)$ converges to the model space in the measured Gromov-Hausdorff sense after removing negligible subsets. 
	
	Previously, Song \cite{song2025entropy} established this type of stability for negatively curved symmetric spaces, where both the $n$-volume of the removed subsets and the $(n-1)$-volume of their boundaries converge to zero. We construct a counterexample demonstrating that this stronger stability notion does not generally hold in the product case; in particular, the condition that the $(n-1)$-volume of the boundary of removed subsets converges to zero cannot be imposed. 

	Nonetheless, we prove that an entropy-minimizing sequence $(M,g_i)$ converges to the model space after removing subsets whose $n$-volume converges to zero in the measured Gromov-Hausdorff topology. This result provides a weaker form of stability compared to the negatively curved symmetric case. A key ingredient in establishing this stability is our proof of the intrinsic uniqueness of the spherical Plateau solution for products of negatively curved symmetric spaces, which is of independent interest.
\end{abstract}

\maketitle

\numberwithin{equation}{section}

\section{Introduction}\label{sec:introduction}
Let $(M,g)$ be a closed Riemannian manifold and let $(\tilde{M},\tilde{g})$ be its universal cover equipped with the pulled-back metric. The volume entropy $h(g)$ of $(M,g)$ is defined as 
\[
	h(g):=\lim_{R\to\infty}\frac{1}{R}\log\mathrm{Vol}(B_{\tilde{g}}(o,R),\tilde{g})
\] where $B_{\tilde{g}}(o,R)$ is the ball of radius $R$ centered at $o\in\tilde{M}$ with respect to the metric $\tilde{g}$. The volume entropy is a geometric invariant of the manifold (i.e., it is independent of the choice of $o\in\tilde{M}$ \cite{Manning:1979aa}.) and has been studied extensively in the literature. One of the most important results is the volume entropy inequality and rigidity theorem for locally symmetric spaces proved by G. Besson, G. Courtois, and S. Gallot \cite{Besson:1995aa,Besson:1996aa} (cf. \cite{Ruan:2024aa} for the Cayley hyperbolic space), which states as the following: let $(M,g_0)$ be a closed $n$-manifold that is locally symmetric with negative curvature at most $-1$. Then for any Riemannian metric $g$ on $M$ with $\mathrm{Vol}(M,g)=\mathrm{Vol}(M,g_0)$, the following inequality holds:
\[
	h(g)\ge h(g_0)
\] with equality if and only if $(M,g)$ is isometric to $(M,g_0)$. While extending the minimal entropy rigidity to all nonpositively curved locally symmetric spaces (\cite[Open Question 5]{Besson:1996aa}) is still open, an important progress has been made by C. Connell and B. Farb \cite{Connell:2003aa} who generalized the minimal entropy rigidity to the case of products of negatively curved symmetric spaces.

\begin{thm}[\textbf{Minimal entropy rigidity for products of negatively curved symmetric spaces} \cite{Connell:2003aa}] 
	\label{thm:intro CF}
	Let $(M,g_0)$ be a closed oriented $n$-manifold which is locally isometric to a product $(X_1^{n_1},g_1)\times\cdots (X_k^{n_k},g_k)$ of negatively curved symmetric spaces (where $(X_j,g_j)$ has maximum sectional curvature $-1$ after scaling and dimension at least $3$). Define the metric $g_{\mathrm{min}}$ on $M$
	\[
		g_{\mathrm{min}}=\alpha_1^2 g_1\times\cdots\times\alpha_k^2 g_k,\, \alpha_j=\frac{h_j}{\sqrt{n_j}}\prod_{l=1}^k \left(\frac{\sqrt{n_l}}{h_l}\right)^{n_l/n},\, h(g_{\mathrm{min}})=\sqrt{n}\prod_{j=1}^{k}\left(\frac{h_j}{\sqrt{n_j}}\right)^{n_j/n},
	\] where
	\begin{align*}
		n_j=\dim X_j,\quad n=\sum_{j=1}^k n_j=\dim M,\quad h_j=\text{the volume entropy of }g_j.
	\end{align*}
	Then for any Riemannian metric $g$ on $M$ with $\mathrm{Vol}(M,g)=\mathrm{Vol}(M,g_{\mathrm{min}})$, the following inequality holds:
	\[
		h(g)\ge h(g_{\mathrm{min}})
	\] 
	with equality if and only if $(M,g)$ is isometric to $(M,g_{\mathrm{min}})$.
\end{thm}

The main goal of this paper is to study the stability of the minimal entropy rigidity for products of negatively curved symmetric spaces. We consider the following question: given a closed oriented $n$-manifold $M$ that is locally isometric to a product of negatively curved symmetric spaces, how stable is the minimal entropy rigidity? 

The question of stability for the minimal entropy rigidity has been studied by various authors, see \cite{Courtois:1999aa,Bessieres:2012aa,Guillarmou:2019aa,Guillarmou:2022aa,butt2022quantitative,song2025entropy}. In particular, the most relevant work to this article is the following stability result for negatively curved symmetric spaces recently proved by A. Song \cite{song2025entropy}, which essentially means that the minimal entropy rigidity for negatively curved locally symmetric spaces is stable after removing a sequence of negligible subsets. Throughout this paper, we will use $\mathrm{Vol}(A,g)$ and $\mathrm{Area}(\partial A,g)$ to denote the $n$-volume of a subset $A$ and the $(n-1)$-volume of its boundary $\partial A$ with the metric $g$, respectively.

\begin{thm}[\textbf{Minimal entropy stability for negatively curved symmetric spaces} \cite{song2025entropy}]
	\label{thm:intro Song}
	 Let $(M,g_0)$ be a closed oriented manifold of dimension at least $3$ where $g_0$ is a negatively curved locally symmetric metric. Let $\{g_i\}$ be a sequence of Riemannian metrics on $M$ with $\mathrm{Vol}(M,g_i)=\mathrm{Vol}(M,g_0)$. If $h(g_i)\to h(g_0)$, then there is a sequence of smooth subsets $Z_i\subset M$ with
	\begin{align*}
		\lim_{i\to\infty}\mathrm{Vol}(Z_i,g_i)=\lim_{i\to\infty}\mathrm{Area}(\partial Z_i,g_i)=0
	\end{align*}
	such that $(M\setminus Z_i,d_{g_i|_{M\setminus Z_i}},d\mathrm{vol}_{g_i})$ converges to $(M,d_{g_0},d\mathrm{vol}_{g_0})$ in the measured Gromov-Hausdorff topology. Here, $d_{g_i|_{M\setminus Z_i}}$ is the distance induced by the metric $g_i$ on $M\setminus Z_i$, i.e., 
	\[
		d_{g_i|_{M\setminus Z_i}}(x,y)=\inf_{\gamma}\mathrm{Length}(\gamma,g_i)
	\] where the infimum is taken over all rectifiable curves connecting $x$ and $y$ in $M\setminus Z_i$.
\end{thm}

In this paper, we investigate the similar type of stability for products of negatively curved symmetric spaces. One may expect that the stability behavior for the product case is similar to the negatively curved symmetric case. However, the product case turns out to be more delicate, as we show by constructing a counterexample that the stability analogous to Theorem \ref{thm:intro Song} does not hold for products of negatively curved symmetric spaces.

\begin{alphathm}\label{thm:intro_counterexample}
	Let $M=\mathbb{H}^n\times \mathbb{H}^n/\Gamma_1\times\Gamma_2,n\ge 3,$ where $\Gamma_1$ and $\Gamma_2$ are cocompact lattices in $\mathrm{Isom}(\mathbb{H}^n)$. Denote by $g_0$ the product metric on $\mathbb{H}^n\times \mathbb{H}^n$.
	Then there exists a sequence of Riemannian metrics $\{g_i\}$ on $M$ with $\mathrm{Vol}(M,g_i)=\mathrm{Vol}(M,g_0)$ satisfying 
	\begin{enumerate}
		\item $h(g_i)\to h(g_0)$, and
		\item there does not exist a sequence of smooth subsets $Z_i\subset M$ with
		\begin{align*}
			\lim_{i\to\infty}\mathrm{Vol}(Z_i,g_i)=\lim_{i\to\infty}\mathrm{Area}(\partial Z_i,g_i)=0,
		\end{align*}
		such that $(M\setminus Z_i,d_{g_i|_{M\setminus Z_i}},d\mathrm{vol}_{g_i})$ converges to $(M,d_{g_0},d\mathrm{vol}_{g_0})$ in the measured Gromov-Hausdorff topology.
	\end{enumerate}
\end{alphathm}

The concrete construction of the example is provided in Section \ref{sec:counterexample}. A key observation used to construct the counterexample is that the volume growth of geodesic balls in the product space is concentrated in a certain direction. This allows one to construct a sequence of metrics that shortens the distance in the other directions while the volume entropy still converges to the minimal entropy (See Figure \ref{fig:comparison}). Such a construction is not possible in the negatively curved symmetric case due to the equidistribution property of geodesic spheres in those spaces (cf. Remark \ref{remark:equidistribution}).

Although it is unexpected that the analogous stability does not hold for the product case, we still obtain the following weaker stability, in which the $n$-volume of the removed subsets $Z_i$ still converges to zero, whereas the $(n-1)$-volume of $\partial Z_i$ may not.

\begin{alphathm}\label{thm:intro_stability of the product}
	Let $(M,g_0)$ be a closed oriented $n$-manifold defined as in Theorem \ref{thm:intro CF}. Suppose that $\{g_i\}$ is a sequence of Riemannian metrics on $M$ with $\mathrm{Vol}(M,g_i)=\mathrm{Vol}(M,g_m)$ satisfying
	\begin{equation*}
		\lim_{i\to\infty}h(g_i)=h(g_m),
	\end{equation*}
	where $g_m=\frac{h(g_{\mathrm{min}})^2}{4n}g_{\mathrm{min}}$ is the normalized metric of $g_{\mathrm{min}}$ defined in Theorem \ref{thm:intro CF}. Then there exists a sequence of smooth subsets $Z_i\subset M$ with
	\begin{equation*}
		\lim_{i\to\infty}\mathrm{Vol}(Z_i,{g_i})=0
	\end{equation*} such that $(M\setminus Z_i,d_{g_i|_{M\setminus Z_i}},d\mathrm{vol}_{g_i})$ converges to $(M,d_{g_m},d\mathrm{vol}_{g_m})$ in the measured Gromov-Hausdorff topology.
\end{alphathm}

Note that Theorem \ref{thm:intro_stability of the product} is optimal in terms of the `size' of the removed subsets $Z_i$. To quantify the strength of the stability, one can consider the notion of the coarse dimension of the sequence of boundaries $\partial Z_i$ suggested in \cite[Remark 3.9]{song2025entropy}. Using this notion, Theorem \ref{thm:intro_stability of the product} can be interpreted as that the minimal entropy rigidity for products of negatively curved symmetric spaces is stable up to a codimension $1$ subset, and it cannot be improved due to Theorem \ref{thm:intro_counterexample}. Indeed, the counterexample implies that the $(n-1)$-volume of $\partial Z_i$ may not converge to zero in general, thus the coarse dimension of $\partial Z_i$ must be at least $n-1$.

We also remark that Theorem \ref{thm:intro_stability of the product} and the stability established by Song (Theorem \ref{thm:intro Song}) do not assume any curvature condition. In the literature, L. Bessi\'eres, G. Besson, G. Courtois, and S. Gallot \cite{Bessieres:2012aa} studied the minimal entropy stability for hyperbolic manifolds under the Ricci curvature lower bound condition in the Gromov-Hausdorff topology without removing subsets. On the other hand, the work of D. Kazaras, A. Song, and K. Xu \cite{kazaras2024scalarcurvaturevolumeentropy} showed that the scalar curvature lower bound does not impose an upper bound on the volume entropy in general, which disproves the conjecture of I. Agol, P. A. Storm, and W. P. Thurston \cite[Conjecture 12.2]{Agol:2007aa}.

Following the previous remarks, it leads to a question whether there is a general relation between curvature lower bound conditions and the size of the removed subsets in the minimal entropy stability. As the Ricci lower bound ensures the stability without removing any subsets and the scalar curvature lower bound does not control the volume entropy effectively, it is reasonable to investigate a potential relation between the coarse dimension of the removed boundaries and the $m$-intermediate curvature lower bound introduced by S. Brendle, S. Hirsch, and F. Johne \cite{Brendle:2024aa}. We leave this question for future research.

Another main result of this paper is the following intrinsic uniqueness of the spherical Plateau solution for products of negatively curved symmetric spaces, which is not only instrumental to the proof of Theorem \ref{thm:intro_stability of the product} but also interesting in its own right.

\begin{alphathm}\label{thm:intro_uniqueness of the spherical plateau solution}
	Let $(M,g_0)$ be a closed oriented $n$-manifold defined as in Theorem \ref{thm:intro CF}. Then any spherical Plateau solution for $M$ is intrinsically isomorphic to $\left(M,\frac{h(g_{\mathrm{min}})^2}{4n}g_{\mathrm{min}}\right)$ where $g_{\textrm{min}}$ is defined as in Theorem \ref{thm:intro CF}.
\end{alphathm}
The notion of spherical Plateau solutions was recently introduced in \cite{song2024sphericalvolumesphericalplateau} and is reviewed in Section \ref{sec:uniqueness of the spherical plateau solution}. Roughly speaking, a spherical Plateau solution for a closed oriented $n$-manifold is an $n$-dimensional minimal submanifold intended to realize the spherical volume of the manifold. The spherical volume is a significant invariant introduced in \cite{Besson:1991aa} and played a crucial role in the proof of the minimal entropy rigidity for negatively curved symmetric spaces \cite{Besson:1995aa,Besson:1996aa}. We also refer the readers to \cite{song2024hyperbolicgroupssphericalminimal,song2024randomminimalsurfacesspheres} for more results on the spherical Plateau solutions.

Using the strong connection between the spherical volume and the minimal entropy as seen in \cite{Besson:1995aa,Besson:1996aa}, Theorem \ref{thm:intro_uniqueness of the spherical plateau solution} and the estimates in its proof are crucially used to prove Theorem \ref{thm:intro_stability of the product}. To prove Theorem \ref{thm:intro_uniqueness of the spherical plateau solution}, the barycenter map method for the product case developed by Connell and Farb \cite{Connell:2003aa} is adapted to our setting. This powerful method was used in all aforementioned fundamental results on the volume entropy. We also refer the readers to \cite{Connell:2003um,Connell:2003vm} for applications of the barycenter method. 

The paper is organized as follows. Basic notations, definitions, and fundamental theorems regarding the integral current spaces and intrinsic flat topology used in this paper are summarized in Section \ref{sec:preliminaries}. In Section \ref{sec:uniqueness of the spherical plateau solution}, we recall the notion of spherical Plateau problem and establish the intrinsic uniqueness of the spherical Plateau solution for products of negatively curved symmetric spaces. The barycenter map and the estimates appeared in this section are nontrivial adaptation of the arguments in \cite{Connell:2003aa} to the product case. In Section \ref{sec:stability of the product}, we prove Theorem \ref{thm:intro_stability of the product} by combining the results from Section \ref{sec:uniqueness of the spherical plateau solution} and the stability argument developed in \cite{song2025entropy}. In Section \ref{sec:counterexample}, we construct a counterexample to the analogous stability for the negatively curved symmetric case.\\

\textbf{Acknowledgments.} I would like to thank Antoine Song for suggesting the problem and for his invaluable guidance and support throughout this research. I am also grateful to Yangyang Li for helpful discussions and comments and Ben Lowe for pointing out the reference \cite{Ruan:2024aa}. I would also like to thank the anonymous referee for their careful reading of the manuscript and for valuable comments that improved the exposition.

\section{Preliminaries}\label{sec:preliminaries}
As the results in this paper are proved using the theory of integral current spaces and the intrinsic flat topology, we briefly recall the basic terminologies and theorems in this section. These concepts are developed in \cite{Ambrosio:2000vw,Wenger:2011aa,Sormani:2011aa}, and we also refer the readers to \cite[Section 1]{song2024sphericalvolumesphericalplateau} for a more detailed summary.

Let $(E,d)$ be a complete metric space. Let $n\ge 0$, and $\mathcal{D}^n(E)$ be the set of $(n+1)$-tuples $(f,\pi_1,\ldots,\pi_n)$ of real valued Lipschitz functions on $E$ where $f$ is assumed to be bounded. Following the theory of currents in Euclidean spaces, a $(n+1)$-tuple $(f,\pi_1,\ldots,\pi_n)\in\mathcal{D}^n(E)$ is denoted by $f\, d\pi_1\wedge\cdots\wedge d\pi_n$. Normally, an $n$-dimensional current $T$ on $E$ is defined as a real-valued multilinear functional on $\mathcal{D}^n(E)$ satisfying locality, continuity, and having finite mass, which we do not write down explicitly. See \cite{Ambrosio:2000vw} or \cite[Definition 2.13]{Sormani:2011aa}. Instead, we simply introduce integer rectifiable and integral currents in metric spaces as parametrized by Lipschitz maps from subsets of Euclidean spaces:

\begin{definition}
	$\bullet$ Let $A\subset\mathbb{R}^n$ and $\theta\in L^1(A;\mathbb{N})$. An elementary current $[\![\theta]\!]$ is defined as
	\begin{align*}
		[\![\theta]\!](f\,d\pi_1\wedge\cdots\wedge d\pi_n):=\int_A\theta\,f\,d\pi_1\wedge\cdots\wedge d\pi_n.
	\end{align*}

	$\bullet $ A multilinear functional $T$ on $\mathcal{D}^n(E)$ is called an $n$-dimensional integer rectifiable current in $E$ if it is of the form
	\begin{align*}
		T=\sum_{i=1}^\infty (\phi_i)_\sharp[\![\theta_i]\!]
	\end{align*}
	where $\phi_i:A_i\to E$ are Lipschitz maps with compact Borel measurable $A_i\subset\mathbb{R}^n$, $\phi_i(A_i)$ are disjoint, and $\theta_i\in L^1(A_i;\mathbb{N})$. The pair $(\{\phi_i:A_i\to E\},\{\theta_i\})$ is called a parametrization of $T$.

	$\bullet$ An integer rectifiable current $T$ is called an integral current if its boundary $\partial T$ is an $(n-1)$-dimensional integer rectifiable current.
\end{definition}

\begin{notation}
	Let $T$ be an $n$-dimensional integral current in $E$. As a generalization of oriented submanifolds, $T$ admits a well-defined notion of volume measure called the mass measure and a notion of total volume called the total mass.
	\begin{itemize}
		\item $||T||$: the mass measure of $T$
		\item $\mathbf{M}(T)$: the total mass of $T$
		\item $\text{spt}(T)$: the support of the measure $||T||$
		\item $\phi_\sharp(T)$: the push-forward of an integral current $T$ by a Lipschitz map $\phi$
		\item $T\llcorner A$: the restriction of an integral current $T$ to a Borel set $A$
		\item $[\![1_M]\!]$: the natural $n$-dimensional integral current induced by integration on a complete oriented Riemannian manifold $(M,g)$ with compact boundary and finite volume
	\end{itemize}
\end{notation}

\begin{definition}[Flat topology]
	Let $T$ and $T_i$ be $n$-dimensional integral currents in $E$. We say that $T_i$ converges to $T$ in the flat topology if there are sequences $\{U_i\}$ and $\{V_i\}$ of $n$- and $(n+1)$-dimensional integral currents in $E$ such that
	\begin{align*}
		T_i-T=U_i+\partial V_i\qquad\text{and}\qquad \lim_{i\to\infty}\mathbf{M}(U_i)=\lim_{i\to\infty}\mathbf{M}(V_i)=0.
	\end{align*}
\end{definition}
\textbf{Lower semi-continuity of mass.} \cite[Section 3]{Ambrosio:2000vw} If $T_i$ converges to $T$ in the flat topology (or weak topology), then 
\begin{align*}
	\mathbf{M}(T)\le\liminf_{i\to\infty}\mathbf{M}(T_i).
\end{align*}

\textbf{The area and coarea formulas.} \cite[Theorem 8.2]{Ambrosio:2000aa}
Let $S$ be a countably $\mathscr{H}^n$-rectifiable set in $E$.
\begin{enumerate}
	\item \textbf{Area formula.} Given a Lipschitz map $g:S\to E'$, the following area formula holds for any Borel function $\theta:S\to[0,\infty]$:
	\begin{align*}
		\int_S\theta(p)\mathbf{J}_n(d^S g_p)\,d\mathscr{H}^n(p)=\int_{E'}\sum_{p\in S\cap g^{-1}(z)}\theta(p)\,d\mathscr{H}^n(z),
	\end{align*}
	where $\mathbf{J}_n(d^S g_p)$ is the $n$-dimensional Jacobian of $g$ at $p$. As the classical area formula, this formula can be used to compute the mass of an integral current by its image under a Lipschitz map.
	\item \textbf{Coarea formula.} Given a Lipschitz map $\pi:S\to\mathbb{R}^k$, $k\le n$, the following coarea formula holds for any Borel function $\theta:S\to[0,\infty]$:
	\begin{align*}
		\int_S\theta(p)\mathbf{C}_k(d^S\pi_p)\,d\mathscr{H}^n(p)=\int_{\mathbb{R}^k}\int_{\pi^{-1}(z)}\theta(p)\,d\mathscr{H}^{n-k}(p)\,d\mathscr{H}^k(z),
	\end{align*}
	where $\mathbf{C}_k(d^S\pi_p)$ is the $k$-dimensional coarea of $\pi$ at $p$. Similar to the classical coarea formula, we can use this formula to compute the mass of an integral current by slicing it with a Lipschitz map.
\end{enumerate}

\textbf{The slicing theorem.} \cite[Theorems 5.6 and 5.7]{Ambrosio:2000vw}
Let $T$ be an $n$-dimensional integral current in $E$ and let $\pi:\mathrm{spt}(T)\to\mathbb{R}^k$, $k\le n$, be a Lipschitz map. Then for almost all $x\in\mathbb{R}^k$, there is an $(n-k)$-dimensional integral current $\langle T,\pi,x\rangle$ in $E$ satisfying that for all $f\,d\pi_1\wedge\cdots\wedge d\pi_{n-k}\in\mathcal{D}^{n-k}(E)$ and $\psi\in C_c(\mathbb{R}^k)$, we have
\begin{align*}
	\int_{\mathbb{R}^k}\left[\langle T,\pi,x\rangle(f\,d\pi_1\wedge\cdots\wedge d\pi_{n-k})\right]\psi(x)\,dx=T(f(\psi\circ\pi)d\pi\wedge d\pi_1\wedge\cdots\wedge d\pi_{n-k})
\end{align*}
where for $\pi=(u_1,\ldots,u_k)$, $d\pi:=du_1\wedge\cdots\wedge du_k$. It allows us to slice an integral current by level sets of a Lipschitz map which are also integral currents.\\

\textbf{Integral current spaces and intrinsic flat topology.} \cite{Sormani:2011aa}
An integral current space is an intrinsic generalization of integral currents. It is a triple $(X,d,S)$ where $(X,d)$ is a metric space and $S$ is an integral current in the completion $\overline{X}$ of $(X,d)$ with $\mathrm{spt}(S)=\overline{X}$. The intrinsic flat topology is defined similarly to the definition of Gromov-Hausdorff topology: two integral current spaces are close if they are isometrically embedded in a common complete metric space in a way that their embedded images are close in the usual flat topology.\\

\textbf{Wenger's compactness theorem.} \cite{Wenger:2011aa,Sormani:2011aa} Let $(X_i,d_i,S_i)$ be a sequence of integral current spaces. If the masses of $S_i$ are uniformly bounded and the diameters of $X_i$ are uniformly bounded, then there is a subsequence converging to an integral current space in the intrinsic flat topology.

\section{Uniqueness of the Spherical Plateau solution for products of negatively curved symmetric spaces}\label{sec:uniqueness of the spherical plateau solution}

In this section, we first review the definition of the spherical volume and spherical Plateau solution for a closed oriented $n$-manifold $M$. We then compute the spherical volume for products of negatively curved symmetric spaces and show the intrinsic uniqueness of the spherical Plateau solution for $M$, which is a key ingredient in proving the stability in Section \ref{sec:stability of the product}. The Jacobian and distance estimates with respect to the barycenter map are the main task in the proof of the uniqueness.

Let $(M,g_0)$ be a closed oriented $n$-manifold that is locally isometric to a product of negatively curved symmetric spaces of dimension at least $3$. In this section, we study the spherical Plateau solutions for $M$ introduced in \cite{song2025entropy}.

Let $\Gamma:=\pi_1(M)$ and denote by $S^\infty$ the unit sphere in the Hilbert space $\ell^2(\Gamma)$. The $\ell^2$-norm induces a Riemannian metric $\mathbf{g}_{\text{Hil}}$ on $S^\infty$. The group $\Gamma$ acts isometrically on $S^\infty$ by the left regular representation $\lambda_\Gamma:\Gamma\to\text{End}(\ell^2(\Gamma))$ defined as
\begin{align*}
	(\lambda_\Gamma(\gamma).f)(x):=f(\gamma^{-1}x) \quad \text{for all }\gamma\in\Gamma, x\in\Gamma, f\in S^\infty.
\end{align*}

Since $\Gamma$ is infinite and torsion-free, the action is free and proper. Thus, the quotient space $S^\infty/\lambda_\Gamma(\Gamma)$ becomes a classifying space for $\Gamma$. 

Also, $S^\infty/\lambda_\Gamma(\Gamma)$ is a Hilbert manifold equipped with the Hilbert-Riemannian metric $\mathbf{g}_{\text{Hil}}$ induced by the $\ell^2$-norm.

Given base points $p_0\in M$ and $q_0\in S^\infty/\lambda_\Gamma(\Gamma)$, there exists a smooth immersion $M\to S^\infty/\lambda_\Gamma(\Gamma)$ inducing the identity map from $\pi_1(M,p_0)$ to $\pi_1(S^\infty/\lambda_\Gamma(\Gamma),q_0)$, which is unique up to homotopies sending $p_0$ to $q_0$. It determines a unique homotopy class of maps; we say such maps to be ``admissible''. Define
\[
	\mathscr{H}_M:=\{\phi:M\to S^\infty/\lambda_{\Gamma}(\Gamma)\,|\, \phi\text{ is an admissible smooth immersion}\}.
\]
Any map $\phi\in\mathscr{H}_M$ defines the pull-back Riemannian metric $\phi^*\mathbf{g}_{\text{Hil}}$ on $M$.


\begin{definition}
	(Spherical volume and Plateau solution).\\
	\indent$\bullet$ \emph{The spherical volume of $M$} is defined as
	\[
		\text{SphereVol}(M):=\inf\{\mathrm{Vol}(M,\phi^{*}(\mathbf{g}_{\mathrm{Hil}}))\,|\, \phi\in\mathscr{H}_M\}.
	\]

	$\bullet$ We call a sequence $\phi_i\in\mathscr{H}_M$ \emph{minimizing} if 
	\begin{align*}
		\lim_{i\to\infty}\mathrm{Vol}(M,\phi_i^{*}(\mathbf{g}_{\mathrm{Hil}}))=\text{SphereVol}(M).
	\end{align*} Denoting by $[\![1_M]\!]$ the integral current in $(M,g_0)$ induced by $M$ and its orientation, we can consider the sequence of integral currents $(\phi_{i})_{\sharp}[\![1_M]\!]$ in $S^\infty/\lambda_\Gamma(\Gamma)$.\vspace{5pt}\\
	\indent$\bullet$ \emph{A spherical Plateau solution for $M$} is defined as the integral current space $C_\infty=(X_\infty,d_\infty,T_\infty)$ which is the limit of a minimizing sequence $(\phi_i)_{\sharp}[\![1_M]\!]$ in the intrinsic flat topology.
	
\end{definition}
Since any minimizing sequence in $\mathscr{H}_M$ has uniformly bounded mass and diameter, Wenger's compactness theorem \cite{Wenger:2011aa} guarantees that there is a subsequential limit $C_\infty$ in the intrinsic flat topology. In what follows, we compute the spherical volume for the product case and show the intrinsic uniqueness of the spherical Plateau solution for $M$.

\subsection{The spherical volume}\label{subsec:spherical volume}
The spherical volume of the manifold $M$ is readily available as part of \cite{Connell:2003aa} using the barycenter map method. We include the proof written in our setting for completeness.

\begin{theorem}\label{thm:spherical volume}
	Let $(M,g_0)$ be a closed oriented $n$-manifold which is locally isometric to a product $(X_1,g_1)\times\cdots (X_k,g_k)$ of negatively curved symmetric spaces (where $(X_i,g_i)$ has maximum sectional curvature $-1$ after scaling and dimension at least $3$). Then we have
	\[
		\mathrm{SphereVol}(M)=\mathrm{Vol}\left(M,\frac{h(g_{\mathrm{min}})^{2}}{4n}g_{\mathrm{min}}\right)
	\]
	where
	\begin{align*}
		g_{\mathrm{min}}&=\alpha_1^2 g_1\times\cdots\times\alpha_k^2 g_k,\, \alpha_i=\frac{h_i}{\sqrt{n_i}}\prod_{i=1}^k \left(\frac{\sqrt{n_i}}{h_i}\right)^{n_i/n},\, h(g_{\mathrm{min}})=\sqrt{n}\prod_{i=1}^{k}\left(\frac{h_i}{\sqrt{n_i}}\right)^{n_i/n}.\\
		&\quad(n_i=\dim X_i,\quad n=\sum_{i=1}^k n_i=\dim M,\quad h_i=\text{the volume entropy of }g_i)
	\end{align*}
\end{theorem}

The above theorem can be obtained by combining the methods from Connell and Farb \cite{Connell:2003aa} and Song \cite{song2025entropy}. Firstly, we define the barycenter map as the following: let $(\tilde{M},\tilde{g}_{\mathrm{min}})$ be the universal cover of $M$ equipped with the pullback metric of $g_{\mathrm{min}}$ and fix a base point $o\in\tilde{M}$. 
Let $\{\nu_x\}_{x\in\tilde{M}}$ be the family of the Patterson-Sullivan measures on $\partial\tilde{M}$. Note that for any $x\in\tilde{M}$, the support of $\nu_x$ is equal to the Furstenberg boundary $\partial_F \tilde{M}$, which in our case is strictly contained in the visual boundary $\partial\tilde{M}$. We refer to \cite{Albuquerque:1999uu} for the definition and properties of those measures defined on the higher rank symmetric spaces. 

For $x\in \tilde{M}$ and $\theta\in\partial\tilde{M}$, let $B_0(x,\theta)$ be the Busemann function on $\tilde{M}$ with respect to the base point $o\in\tilde{M}$ defined by
\[
	B_0(x,\theta)=\lim_{t\to\infty}d_{\tilde{g}_{\mathrm{min}}}(x,\gamma_{\theta}(t))-t
\] where $d_{\tilde{g}_{\mathrm{min}}}$ is the distance on $\tilde{M}$ with respect to $\tilde{g}_{\mathrm{min}}$ and $\gamma_{\theta}$ is the unique geodesic ray with $\gamma_{\theta}(0)=o$ and $\gamma_{\theta}(\infty)=\theta$. For $f\in S^\infty$, consider the function $\mathcal{B}_f:\tilde{M}\to [0,\infty]$ defined by
\[
	\mathcal{B}_f(x):=\sum_{\gamma\in\Gamma}|f(\gamma)|^2 \int_{\partial_F\tilde{M}}B_0(x,\theta)\, d\nu_{\gamma.o}(\theta)
\]
where $\gamma.o$ is the monodromy action of $\gamma$ on $o\in\tilde{M}$. 

Let $\mathbb{S}^{+}$ be the set of finitely supported functions in $S^\infty$. We define the barycenter map $\mathrm{Bar}:\mathbb{S}^{+}\to\tilde{M}$ by
\begin{equation}\label{eqn:barycenter map}
	\mathrm{Bar}(f):=\text{ the unique point minimizing }\mathcal{B}_f.
\end{equation} The map $\mathrm{Bar}$ is well-defined on $\mathbb{S}^+$ and independent of the choice of base points as shown in \cite[Section 4]{Connell:2003aa} by verifying that $\mathcal{B}_f$ is strictly convex and $\mathcal{B}_f(x)\to\infty$ as $x\to\infty$. It is also straightforward that $\mathrm{Bar}$ is $\Gamma$-equivariant, thus, the quotient map of $\mathrm{Bar}$ is also well-defined on $\mathbb{S}^{+}/\lambda_\Gamma(\Gamma)$ into $M$, which we still denote by $\mathrm{Bar}$.

We now state the Jacobian estimate of the barycenter map in our setting.


\begin{lemma}\label{lemma:jacobian bound} (\cite[Theorem 5.1]{Connell:2003aa}). 
	Let $f\in\mathbb{S}^{+}$ and $Q$ be the tangent $n$-plane at $f$ of a totally geodesic $n$-simplex in $\mathbb{S}^{+}$ passing through $f$. Then, $\mathrm{Bar}$ is differentiable along $Q$ and we have
	\begin{equation}\label{eqn:Jac_est}
		\left|\mathrm{Jac}\,\mathrm{Bar}|_{Q}\right|\le\left(\frac{4n}{h(g_{\mathrm{min}})^2}\right)^{n/2}.
	\end{equation}
\end{lemma}

The inequality \eqref{eqn:Jac_est} follows from the argument in \cite[Section 5]{Connell:2003aa}, by replacing the probability measure $\mu_y^s$ on $\tilde{N}$ by $\sum_{\gamma\in\Gamma}|f(\gamma)|^2\delta_{\gamma.o}$ on $\tilde{M}$, where $\delta_{\gamma.o}$ stands for the Dirac measure at $\gamma.o$.  We include the proof below for reader's convenience.

\begin{proof}[Proof of Lemma \ref{lemma:jacobian bound}]
	Consider $f\in\mathbb{S}^{+}$. The barycenter $x=\mathrm{Bar}(f)$ is determined by the equation
	\begin{equation}\label{eqn:critical point of bar}
		\sum_{\gamma\in\Gamma}f^2(\gamma)\int_{\partial_F \tilde{M}}d_{(x,\theta)}B_0(\cdot)\, d\nu_{\gamma.o}=0\in T_{x}^{*}\tilde{M}.
	\end{equation} 
	Let $H_f$ and $K_f$ be the endomorphisms of $T_x\tilde{M}$ determined by the following: for any $v,w\in T_x\tilde{M}$
	\begin{align*}
		\tilde{g}_{\mathrm{min}}(H_f(v),w)&:=\sum_{\gamma\in\Gamma}f^2(\gamma)\int_{\partial_F \tilde{M}} \tilde{g}_{\mathrm{min}}(d_{(x,\theta)}B_0(v),d_{(x,\theta)}B_0(w))\, d\nu_{\gamma.o}(\theta),\\
		\tilde{g}_{\mathrm{min}}(K_f(v),w)&:=\sum_{\gamma\in\Gamma}f^2(\gamma)\int_{\partial_F \tilde{M}}\mathrm{Hess}_{(x,\theta)}B_0(v,w)\, d\nu_{\gamma.o}(\theta).
	\end{align*}
	We also denote the corresponding symmetric bilinear forms by $H_f$ and $K_f$, which are both positive semi-definite. Note also that $\mathrm{tr}_{\tilde{g}_{\mathrm{min}}}H_f=1$ since $|dB_0|_{\tilde{g}_{\mathrm{min}}}=1$.
	Let $Q$ be the tangent $n$-plane of a totally geodesic $n$-simplex in $\mathbb{S}^{+}$ passing through $f$. By differentiating \eqref{eqn:critical point of bar} with respect to $f$, we get for all tangent vectors $\dot{f}\in Q$,
	\begin{align*}
		\sum_{\gamma\in\Gamma}2f(\gamma)\dot{f}(\gamma)&\int_{\partial_F \tilde{M}}d_{(x,\theta)}B_0(\cdot)\, d\nu_{\gamma.o}(\theta)\\
							       &+\sum_{\gamma\in\Gamma} f^2(\gamma)\int_{\partial_F \tilde{M}}\mathrm{Hess}_{(x,\theta)}B_0(d\mathrm{Bar}|_{Q}(\dot{f}),\cdot)\, d\nu_{\gamma.o}(\theta)=0\in T_x^* \tilde{M}
	\end{align*}
	By Cauchy-Schwarz inequality, we have for $v\in T_{x}\tilde{M}$ and $\dot{f}\in Q$ with $||\dot{f}||_{\ell^2}=1$,
	\begin{equation}\label{eqn:K and F}
		\begin{split}
			K_f(d\mathrm{Bar}|_Q(\dot{f}),v)&\le 2||\dot{f}||_{\ell^2}\left[\sum_{\gamma\in\Gamma}f^2(\gamma)\left(\int_{\partial_F \tilde{M}} d_{(x,\theta)}B_0(v)\, d\nu_{\gamma.o}(\theta)\right)^2\right]^{1/2}\\
						       &\le 2[H_f(v,v)]^{1/2},
		\end{split}
	\end{equation} and this implies that
	\begin{equation}\label{eqn:jacobian computation1}
		\left|\mathrm{Jac Bar}|_Q\right|\le 2^n\frac{(\det H_f)^{1/2}}{\det K_f}.
	\end{equation}
	By the hypothesis on $\tilde{M}$, it is shown as in the proof of \cite[Theorem 5.1]{Connell:2003aa} that 
	\[
		\partial_F \tilde{M}\cong \partial X_{1}\times\cdots\times\partial X_{k}.
	\] Let $B_i$ denote the Busemann function for the rank one symmetric space $X_i$ with metric $g_i$. By using the product coordinates, we have the following formulas: given 
	$x=(x_1,\ldots,x_k)\in X_1\times\cdots\times X_k=\tilde{M}$ and $\theta=(\theta_1,\ldots,\theta_k)\in\partial X_1\times\cdots\times \partial X_k=\partial_F\tilde{M}$,
	\begin{align*}
		B_0(x,\theta)&=\sum_{i=1}^{k}\frac{\alpha_i}{\sqrt{k}}B_{i}(x_i,\theta_i),\\
		\nabla_{x} B_0&=\sum_{i=1}^{k}\frac{1}{\alpha_i\sqrt{k}}\nabla_{x_i}^{g_i}B_i, \text{ hence, }|\nabla B_0|=1,\\
		\mathrm{Hess}_{x} B_0&=\bigoplus_{i=1}^{k}\frac{1}{\alpha_i\sqrt{k}}\mathrm{Hess}_{x_i}^{g_i}B_i.
	\end{align*} 
	Applying those formulas and using linear algebra techniques, we observe
	\begin{align*}
		\det H_f&=\det\left(\sum_{\gamma\in\Gamma}f^2(\gamma)\int_{\partial_F \tilde{M}}\left|\sum_{i=1}^{k}(\alpha_i\sqrt{k})^{-1}\pi_i^*d_{(x_i,\theta_i)}B_i(\cdot)\right|_{g_{i}}^2\, d\nu_{\gamma.o}(\theta)\right)\\
			&\le\prod_{i=1}^k(\alpha_i\sqrt{k})^{-2n_i}\det\left(\sum_{\gamma\in\Gamma}f^2(\gamma)\int_{\partial X_i}\left|\pi_i^*d_{(x_i,\theta_i)}B_i(\cdot)\right|_{g_{i}}^2\, d((\pi_i)_*\nu_{\gamma.o})(\theta_i)\right),
	\end{align*}
	and
	\begin{align*}
		\det K_f&=\det \left(\sum_{\gamma\in\Gamma}f^2(\gamma)\int_{\partial_F \tilde{M}}\bigoplus_{i=1}^k (\alpha_i\sqrt{k})^{-1} \pi_i^*\mathrm{Hess}_{(x_i,\theta_i)}B_i\, d\nu_{\gamma.o}(\theta)\right)\\
			&=\prod_{i=1}^{k}(\alpha_i\sqrt{k})^{-n_i}\det\left(\sum_{\gamma\in\Gamma}f^2(\gamma)\int_{\partial X_i}\pi_i^{*}\mathrm{Hess}_{(x_i,\theta_i)}B_i\, d((\pi_i)_{*}\nu_{\gamma.o})(\theta_i)\right).
	\end{align*}
	where $\pi_i$ is the projection map of $\partial_F \tilde{M}$ onto $\partial X_i$. Combining them, we get
	\begin{align*}
		\frac{(\det H_f)^{1/2}}{\det K_f}&\le\prod_{i=1}^{k}\frac{\left[\det\left(\sum_{\gamma\in\Gamma}f^2(\gamma)\int_{\partial X_i}|\pi_i^*d_{(x_i,\theta_i)}B_i(\cdot)|^2_{g_i}\, d((\pi_i)_*\nu_{\gamma.o})\right)\right]^{1/2}}{\det \left(\sum_{\gamma\in\Gamma}f^2(\gamma)\int_{\partial X_i} \pi_i^* \mathrm{Hess}_{(x_i,\theta_i)}B_i\, d((\pi_i)_*\nu_{\gamma.o})\right)}\\
						 &=\prod_{i=1}^{k}\frac{(\det H_i)^{1/2}}{\det K_i}
	\end{align*} where $H_i$ and $K_i$ are defined by
	\begin{align*}
		H_i(v,v)&=\sum_{\gamma\in\Gamma}f^2(\gamma)\int_{\partial X_i}|\pi_i^*d_{(x_i,\theta_i)}B_i(v)|^2_{g_i}d((\pi_i)_*\nu_{\gamma.o}),\\
		K_i(v,v)&=\sum_{\gamma\in\Gamma}f^2(\gamma)\int_{\partial X_i} \pi_i^* \mathrm{Hess}_{(x_i,\theta_i)}B_i(v,v)d((\pi_i)_*\nu_{\gamma.o}).
	\end{align*} 
	Note that \eqref{eqn:K and F} holds for each $H_i$ and $K_i$; indeed, we have for $v_i\in T_{x_i}X_i$ and $\dot{f}\in Q$ with $||\dot{f}||_{\ell^2}=1$,
	\begin{equation}\label{eqn:K and F projected}
		K_i(d\mathrm{Bar}|_{Q}(\dot{f}),v_i)\le 2[H_i(v_i,v_i)]^{1/2}.
	\end{equation} 
	
	By the classical fact in \cite{Besson:1991aa}, for each factor $X_i$ that is not a Cayley hyperbolic space, there exist orthogonal endomorphisms $J_1,\ldots,J_{d_i-1}$ at each point defining the complex, quaternionic structure (where $d$ is the real dimension of the division algebra) such that
	\begin{equation}\label{eqn:formula for K and H}
		\begin{split}
			K_i=\mathrm{Id}&-H_i-\sum_{l=1}^{d-1}J_l H_i J_l\in\mathrm{End}(T_{x_i} X_i)\\
			&\text{ (if }X_i\text{ is the real hyperbolic space}, K_i=\mathrm{Id}-H_i).
		\end{split}
	\end{equation}
	When $X_i$ is the Cayley hyperbolic space $\mathbb{O}H^2$, the endomorphisms $J_l$ are not globally well-defined due to the non-associativity of the octonions. However, as shown by Ruan \cite{Ruan:2024aa}, the relation \eqref{eqn:formula for K and H} can be replaced by the analogous identity established for the Cayley case, and the subsequent Jacobian estimate remains valid.

	Since $\mathrm{tr}_{g_i}H_i=1$, by \cite[Lemma 5.5]{Besson:1996aa} (and \cite{Ruan:2024aa} for the Cayley case), we have
	\begin{equation}
		\frac{(\det H_i)^{1/2}}{\det\left(\mathrm{Id}-H_i-\sum_{l=1}^{d-1}J_l H_i J_l\right)}\le\left(\frac{\sqrt{n_i}}{h_i}\right)^{n_i}
	\end{equation} and equality holds if and only if $H_i=\frac{1}{n_i}\mathrm{Id}$. Applying this to \eqref{eqn:jacobian computation1}, it follows that
	\begin{equation}
		|\mathrm{Jac Bar}|_Q|\le 2^n\prod_{i=1}^{k}\left(\frac{\sqrt{n_i}}{h_i}\right)^{n_i}=\left(\frac{4n}{h(g_{\mathrm{min}})^2}\right)^{n/2},
	\end{equation} where $h_i$ is the volume entropy of $(X_i,g_i)$; it concludes the inequality \eqref{eqn:Jac_est}.

\end{proof}

Now we can compute the spherical volume for such manifold $M$.
\begin{proof}[Proof of Theorem \ref{thm:spherical volume}]
	By the standard approximation argument, for any $\phi\in\mathscr{H}_M$ and the integral current $C=\phi_{\sharp}[\![1_M]\!]$ in $S^\infty/\lambda_\Gamma(\Gamma)$, there is a polyhedral chain $P$ of dimension $n$ such that the total mass $\mathbf{M}(P)$ is arbitrarily close to $\mathbf{M}(C)(=\mathrm{Vol}(M,\phi^{*}\mathbf{g}_{\mathrm{Hil}}))$ and $\mathrm{spt}(P)\subset\mathbb{S}^{+}/\lambda_\Gamma(\Gamma)$. Then, for almost every $q\in\mathrm{spt}(P)$, the Jacobian of the map $\mathrm{Bar}$ is well-defined and satisfies the inequality \eqref{eqn:Jac_est}; this implies by the area formula that
	\[
		\mathbf{M}(P)\ge\left(\frac{h(g_{\mathrm{min}})^2}{4n}\right)^{n/2}\mathrm{Vol}(M,g_{\mathrm{min}}).
	\] Thus, by the definition of the spherical volume, we have
	\[
		\textrm{SphereVol}(M)\ge\textrm{Vol}\left(M,\frac{h(g_{\textrm{min}})^2}{4n}g_{\textrm{min}}\right).
	\] The reverse inequality is known for the general case in \cite[Corollary 3.13]{Besson:1991aa}, i.e., given a closed oriented Riemannian $n$-manifold $(M,g)$, the following inequality holds:
	\[
		\textrm{SphereVol}(M)\le\mathrm{Vol}\left(M,\frac{h(g)^2}{4n}g\right).
	\] This concludes the proof of Theorem \ref{thm:spherical volume}.
\end{proof}

\subsection{Intrinsic Uniqueness of the Spherical Plateau Solution}
It may be difficult to establish the general uniqueness of the spherical Plateau solution as integral current spaces. Instead, we recall the notion of intrinsic isomorphism introduced in \cite{song2024sphericalvolumesphericalplateau} and show that any spherical Plateau solution for the product case is intrinsically isomorphic to the product manifold up to scaling.

\begin{definition}
	Let $(N,d_{g_N},[\![1_N]\!])$ be the integral current space induced by a Riemannian manifold $(N,g)$. The integral current spaces $(X,d,T)$ and $(N,d_{g_N},[\![1_N]\!])$ are said to be intrinsically isomorphic if there exists an isometry
	\begin{align*}
		\varphi:(N,d_{g_N})\to (X,L_d)
	\end{align*}
	such that $(\mathrm{id}_X\circ\varphi)_{\sharp}[\![1_N]\!]=T$, where $L_d$ is the intrinsic metric on $X$ induced by $d$. We call such a map $\varphi$ an intrinsic isomorphism.
\end{definition}
Note that the intrinsic isomorphism does not produce an equivalence relation on the set of integral current spaces, but it is a useful tool to compare the intrinsic structures of two integral current spaces especially when one of them is induced by a Riemannian structure.



\begin{theorem}\label{thm:uniqueness}
	Let $(M^n,g_0)$ be a closed oriented Riemannian manifold which is locally isometric to a product $(X_1^{n_1},g_1)\times\cdots (X_k^{n_k},g_k)$ of negatively curved symmetric spaces (where $(X_i^{n_i},g_i)$ has maximum sectional curvature $-1$ and dimension $n_i$ at least $3$). Then any spherical Plateau solution for $M$ is intrinsically isomorphic to $\left(M,\frac{h(g_{\mathrm{min}})^2}{4n}g_{\mathrm{min}}\right)$ where $g_{\textrm{min}}$ is defined as in Theorem \ref{thm:spherical volume}.
\end{theorem}


We first show the following lemma saying that when the Jacobian bound of the barycenter map is almost saturated, the differential of the map becomes close to a linear isometry up to constant scaling. This lemma will be used to construct the limit of barycenter maps defined on a minimizing sequence.

\begin{lemma}\label{lemma:differential and length approx}
	Let $f\in\mathbb{S}^{+}$ and $Q$ be the tangent $n$-plane at $f$ of a totally geodesic $n$-simplex in $\mathbb{S}^{+}$ passing through $f$. For any $\eta>0$ small enough, there exists $c_{\eta}>0$ with $\lim_{\eta\to 0}c_\eta=0$ such that the following holds: if
	\[
		\left|\mathrm{Jac}\,\mathrm{Bar}|_{Q}\right|\ge\left(\frac{4n}{h(g_{\mathrm{min}})^2}\right)^{n/2}-\eta,
	\]
	then for any unit tangent vector $\vec{u}\in Q$,
	\begin{equation}\label{eqn:differential_est}
		|d\mathrm{Bar}|_{Q}(\vec{u})|\ge \left(\frac{4n}{h(g_{\mathrm{min}})^2}\right)^{1/2}-c_{\eta}
	\end{equation} and for any connected continuous piecewise geodesic curve $\alpha\subset\mathbb{S}^+$ of length less than $\eta$ starting at $f$, we have
	\begin{equation}\label{eqn:length_est}
		\mathrm{length}_{g_{\mathrm{min}}}(\mathrm{Bar}(\alpha))\le \left(\left(\frac{4n}{h(g_{\mathrm{min}})^2}\right)^{1/2}+c_{\eta}\right)\mathrm{length}(\alpha).
	\end{equation} 
\end{lemma}

\begin{proof}
	These inequalities can be obtained by applying the argument in \cite[Section 7]{Besson:1995aa} to each $H_i$ defined in the previous lemma and combining them in the product coordinates. We include the proof for completeness.

	Let $0\le\mu^{i}_{1}(f)\le\cdots\le\mu_{n_i}^{i}(f)\le 1$ be the eigenvalues of each $H_i$. Then by \cite[Proposition B.5]{Besson:1995aa}, there exist constants $A_i>0$ such that for each $i=1,\ldots,k,$
	\[
		\frac{(\det H_i)^{1/2}}{\det K_i}\le \left(\frac{\sqrt{n_i}}{h_i}\right)^{n_i}\left(1-A_i\sum_{j=1}^{n_i}(\mu^i_j(f)-\frac{1}{n_i})^2\right).
	\] Therefore, if $|\mathrm{Jac Bar}|_Q|\ge\left(\frac{4n}{h(g_{\mathrm{min}})^2}\right)^{n/2}-\eta$ for small $\eta>0$, then there exists $\eta'>0$ depending only on $\eta$ such that $\lim_{\eta\to 0}\eta'=0$ and 
	\[
		\mu^i_{n_i}(f)\le\frac{1}{n_i}+\eta' \text{ for all }i=1,\ldots,k.
	\] Moreover, it follows from \eqref{eqn:K and F} that for any unit vector $\vec{u}\in Q$,
	\begin{equation}\label{eqn:differential computation}
		|d\mathrm{Bar}|_Q(\vec{u})|\le \left(\frac{4n}{h(g_{\mathrm{min}})^2}\right)^{1/2}+c'_{\eta}
	\end{equation}
	where $\lim_{\eta\to 0}c'_\eta=0$. Combining the above two inequalities, we have for any unit tangent vector $\vec{u}\in Q$
	\begin{equation}
		|d\mathrm{Bar}|_Q(\vec{u})|\ge \left(\frac{4n}{h(g_{\mathrm{min}})^2}\right)^{1/2}-c_{\eta}
	\end{equation} where $c_\eta>0$ and $\lim_{\eta\to 0}c_{\eta}=0$, which shows the first inequality \eqref{eqn:differential_est}.
	To prove the second inequality \eqref{eqn:length_est}, we need the following two estimates: 
	\begin{claim}\label{claim1}(\cite[Lemma 7.5a]{Besson:1995aa}). 
		Let $\kappa>0$ and let $\alpha\subset\mathbb{S}^{+}$ be a connected continuous piecewise geodesic curve. Suppose that for each $f\in\alpha, i=1,\ldots,k$,
		\[
			\mu^i_{n_i}(f)\le 1-\kappa/2.
		\]
		Then
		\[
			\mathrm{length}_{\tilde{g}_{\mathrm{min}}}(\mathrm{Bar}(\alpha))\le K_1\mathrm{length}(\alpha)
		\] for a constant $K_1>0$ depending only on $\kappa$.
	\end{claim}
	
	\begin{proof}[Proof of Claim \ref{claim1}]
		Given $f\in\alpha$, let $V$ be the tangent $1$-plane of $\alpha$ at $f$ and $\dot{f}\in V$ with $||\dot{f}||_{\ell^2}=1$ and $d_f\mathrm{Bar}(\dot{f})\ne 0$. Set $v:=\frac{d_f\mathrm{Bar}(\dot{f})}{|d_f\mathrm{Bar}(\dot{f})|}=(v_1,\ldots,v_n)\in T_x\tilde{M}$ where $v_i=d_{x}\pi_i(v)\in T_{x_i}X_i$ for $\pi_i(x)=x_i$.
		By the hypothesis on $\mu_{n_i}^i$, \eqref{eqn:K and F projected} and \eqref{eqn:formula for K and H}, we have
		\begin{align*}
			&|d_x\pi_i(d\mathrm{Bar}|_Q(\dot{f}))|\le 2\frac{[H_i(v_i,v_i)]^{1/2}}{K_i(v_i,v_i)}\text{, thus, }\\
			&|d\mathrm{Bar}|_Q(\dot{f})|\le 2\sqrt{\sum_{i=1}^{k}\frac{H_i(v_i,v_i)}{[1-H_i(v_i,v_i)-\sum_{l=1}^{d-1}J_lH_iJ_l(v_i,v_i)]^2}}\le K_1
		\end{align*} for some constant $K_1>0$ depending only on $\kappa$. By integrating it, the desired estimate follows.

	\end{proof}

	\begin{claim}\label{claim2}(\cite[Lemma 7.5.b]{Besson:1995aa}). 
		Let $f,f'\in\mathbb{S}^{+}$ and let $\beta$ be the geodesic segment joining $\mathrm{Bar}(f)$ and $\mathrm{Bar}(f')$. Let $P$ be the parallel transport from $\mathrm{Bar}(f)$ to $\mathrm{Bar}(f')$ along $\beta$. Then
		\[
			|H_{f'}\circ P-H_f|_{\tilde{g}_{\mathrm{min}}}\le K_2(\mathrm{length}_{\tilde{g}_{\mathrm{min}}}(\beta)+||f-f'||_{\ell^2})
		\] for a constant $K_2>0$.
	\end{claim}
	
	\begin{proof}[Proof of Claim \ref{claim2}] 
		Given $x=\mathrm{Bar}(f)$ and $x'=\mathrm{Bar}(f')$, let $Y$ be a unit tangent vector in $T_x\tilde{M}$ and $Y'$ be the parallel transport of $Y$ at $x'$ along $\beta$. Then we have
		\begin{align*}
			&|H_{f'}(Y',Y')-H_f(Y,Y)|\\
			&\qquad\le \left|\sum_{\gamma\in\Gamma}f^2(\gamma)\left(\int_{\partial_F \tilde{M}}\left|d_{(x,\theta)}B_0(Y')\right|_{\tilde{g}_{\mathrm{min}}}^2-\left|d_{(x,\theta)}B_0(Y)\right|_{\tilde{g}_{\mathrm{min}}}^2\, d\nu_{\gamma.o}(\theta)\right)\right|\\
			&\qquad\qquad+\left|\sum_{\gamma\in\Gamma}((f')^2(\gamma)-f^2(\gamma))\left(\int_{\partial_F \tilde{M}}\left|d_{(x,\theta)}B_0(Y')\right|_{\tilde{g}_{\mathrm{min}}}^2\, d\nu_{\gamma.o}(\theta)\right)\right|\\
			&\qquad\le K_2(\mathrm{length}_{\tilde{g}_{\mathrm{min}}}(\beta)+||f'-f||_{\ell^2}).
		\end{align*}
		Here, we used the facts that the Hessian of $B_0$ is uniformly bounded and the norm of $dB_0$ is equal to $1$.
	\end{proof}

	To obtain the length estimate \eqref{eqn:length_est}, let $\alpha\subset\mathbb{S}^{+}$ be a continuous piecewise geodesic curve that joins $f$ and $f'$ and has length less than some constant $\eta>0$. Then for any fixed $\kappa>0$, the following holds: for each $i=1,\ldots,k,$
	\[
		\text{if }\mu^i_{n_i}(f)\le 1-\kappa,\text{ then }\mu^i_{n_i}(f')\le 1-\kappa+c'_{\eta}
	\] where $c'_\eta$ depends only on $\eta$ and $\lim_{\eta\to 0}c'_\eta=0$. Indeed, one can apply the same argument in \cite[Lemma 2.4]{song2024sphericalvolumesphericalplateau} to each $\mu_{n_i}^i$ and simply pick the maximum of those constants. By using this, we get
	\[
		|d\mathrm{Bar}(\alpha')|\le \left(\frac{4n}{h(g_{\mathrm{min}})^2}\right)^{1/2}+c''_\eta
	\] where $\alpha'$ is the unit tangent vector of $\alpha$, and $c''_\eta$ is constant depending only on $\eta>0$ such that $\lim_{\eta\to 0}c''_\eta=0$. By choosing $\eta$ small enough and integrating the above inequality, we obtain the inequality \eqref{eqn:length_est}.

\end{proof}

The uniqueness theorem (Theorem \ref{thm:uniqueness}) follows by the same argument as in \cite[Theorem 2.6]{song2025entropy}; their proof is generally applicable for the case that admits the barycenter map defined on $\mathbb{S}^+/\lambda_\Gamma(\Gamma)$ satisfying Lemmas \ref{lemma:jacobian bound} and \ref{lemma:differential and length approx}.
\begin{proof}[Proof of Theorem \ref{thm:uniqueness}]
	Let $\{\phi_i\}\subset\mathscr{H}_M$ be a minimizing sequence, i.e.,
	\begin{align*}
		\lim_{i\to\infty} \mathrm{Vol}(M,\phi_i^{*}(\mathbf{g}_{\mathrm{Hil}}))=\mathrm{SphereVol}(M)=\mathrm{Vol}\left(M,\frac{h(g_{\mathrm{min}})^2}{4n}g_{\mathrm{min}}\right).
	\end{align*} 
	Suppose that the integral currents $C_i:=(\phi_i)_{\sharp}[\![1_M]\!]$ converge in the intrinsic flat topology to an integral current space
	\begin{align*}
		C_\infty=(X_\infty,d_\infty,S_\infty).
	\end{align*}
	We can assume without loss of generality that
	\begin{itemize}
		\item for all $i$ and $y\in M$, any lift of $\phi_i(y)\in S^\infty/\lambda_\Gamma(\Gamma)$ in $S^\infty$ has finite support, i.e., $\text{spt}(C_i)\subset\mathbb{S}^+/\lambda_\Gamma(\Gamma)$, and
		\item $C_i$ is a polyhedral chain, i.e., a finite union of embedded totally geodesic $n$-simplices in $\mathbb{S}^+/\lambda_\Gamma(\Gamma)$. (cf. \cite[Lemma 1.6]{song2025entropy})
	\end{itemize}
	The above perturbation is useful since any interior point of a totally geodesic $n$-simplex in $\mathbb{S}^+/\lambda_\Gamma(\Gamma)$ admits a well-defined tangent $n$-plane, and the barycenter map is differentiable at such points. 

	For convenience, let us denote $g_m=\frac{h(g_{\mathrm{min}})^2}{4n}g_{\mathrm{min}}$. In the remaining proof, Jacobians, lengths, and distances are computed with respect to the metric $g_m$. Fix $o\in\tilde{M}$ and let $\text{Bar}:\mathbb{S}^+/\lambda_\Gamma(\Gamma)\to M$ be the barycenter map defined in \eqref{eqn:barycenter map}. By $\Gamma$-equivariance of $\text{Bar}$ and $\phi_i\in\mathscr{H}_M$, it follows that for any $i$, the restriction $\text{Bar}|_{\text{spt}(C_i)}$ is a Lipschitz homotopy equivalence and
	\begin{equation}\label{eqn:pushforward of barycenter map}
		\text{Bar}_{\sharp}(C_i)=[\![1_M]\!].
	\end{equation}
	By lower semicontinuity of the mass under intrinsic flat convergence \cite{Sormani:2011aa}, we have
	\begin{equation}
		\mathbf{M}(C_\infty)\le\lim_{i\to\infty}\mathbf{M}(C_i)=\mathrm{Vol}(M,g_m).
	\end{equation}
	For any point $q$ in the interior of a totally geodesic $n$-simplex in $\text{spt}(C_i)$, the $n$-dimensional Jacobian of $\text{Bar}$ along the tangent $n$-plane of $\text{spt}(C_i)$ is well-defined and is bounded from above by $1$ with respect to the metric $g_m$ by Lemma \ref{lemma:jacobian bound}. This implies, by the area formula and \eqref{eqn:pushforward of barycenter map}, that
	\begin{align*}
		\mathbf{M}(C_i)\ge \int_{\mathrm{spt}(C_i)} \big|\mathrm{Jac}(\mathrm{Bar}|_{\mathrm{spt}(C_i)})\big|\, d\mathscr{H}^n=\mathrm{Vol}(M,g_m).
	\end{align*}
	It then follows that there are open subsets $\Omega_i$ in the smooth part of $\text{spt}(C_i)$ such that
	\begin{equation}\label{eqn:property of omegas}
		\begin{split}
			&\lim_{i\to\infty}\mathbf{M}(C_i\llcorner\Omega_i)=\lim_{i\to\infty}\mathbf{M}(C_i)=\text{SphereVol}(M),\\
			&\lim_{i\to\infty}||\mathrm{Jac}(\mathrm{Bar}|_{\mathrm{spt}(C_i)})-1||_{L^\infty(\Omega_i)}=0.
		\end{split}
	\end{equation}
	Note that each $\Omega_i$ may not be an integral current since we have no control on its boundary. Instead, we consider, for $r>0$,
	\[
		\Omega_{i,r}:=\text{the }r\text{-sublevel set of the smoothed out distance function from }\Omega_i\text{ in }\mathbb{S}^+/\Gamma\subset S^\infty/\lambda_\Gamma(\Gamma).
	\]

	Using \eqref{eqn:property of omegas}, the coarea formula, and the slicing theorem, there exists $r^{(i)}\in (0,1)$ such that for each $i$, $D_i:=C_i\llcorner\Omega_{i,r^{(i)}}$ is an integral current, and $\text{spt}(D_i)$ is a compact piecewise smooth embedded submanifold of $S^\infty/\lambda_\Gamma(\Gamma)$ satisfying:
	\begin{itemize}
		\item the boundary of $\text{spt}(D_i)$ is piecewise smooth (since we used the smoothed out distance function to construct $\Omega_{i,r^{(i)}}$), and we have
		\begin{equation}
			\lim_{i\to\infty}\mathbf{M}(\partial D_i)=0.
		\end{equation}
		\item after taking a subsequence, $D_i$ still converges to $C_\infty=(X_\infty,d_\infty,S_\infty)$ in the intrinsic flat topology. This implies that there exist a Banach space $(\mathbf{Z}',d)$ and isometric embeddings
		\begin{align*}
			\text{spt}(D_i)\hookrightarrow\mathbf{Z}',\quad \text{spt}(S_\infty)\hookrightarrow\mathbf{Z}',
		\end{align*}
		such that $D_i$ converges to $S_\infty$ in the flat topology inside $\mathbf{Z}'$. From now on, we identify $\text{spt}(D_i)$ and $\text{spt}(S_\infty)$ with their images in $\mathbf{Z}'$.
	\end{itemize}
	
	Let $N_i:=\text{spt}(D_i)$ in $\mathbf{Z}'$. By construction, it follows that
	\begin{itemize}
		\item $N_i$ endowed with the intrinsic metric induced by the metric $d$ in $\mathbf{Z}'$ is a compact oriented Riemannian manifold $(N_i,h_i)$ with a piecewise smooth Riemannian metric $h_i$, possibly with nonempty piecewise smooth boundary.
		\item $\lim_{i\to\infty}\mathrm{Area}(\partial N_i,h_i)=0$.
	\end{itemize}

	In addition, we verify the following (cf. \cite[Assumption 1.3]{song2025entropy}):
	\begin{enumerate}[label=(\alph*)]
		\item the maps $\text{Bar}|_{N_i}:(N_i,d|_{N_i})\to (M,g_m)$ are $C^1$ on the smooth part of $N_i$ and $\lambda$-Lipschitz for some $\lambda>0$ independent of $i$, such that $\text{Bar}_{\sharp}(D_i)$ converges to $[\![1_M]\!]$ in the flat topology inside $(M,g_m)$.
		\item $h_i$ is smooth in $\Omega_i$ and the following holds:
		\begin{align*}
			\lim_{i\to\infty}\mathrm{Vol}(N_i\setminus\Omega_i,h_i)=0,\quad \lim_{i\to\infty}\mathrm{Vol}(\Omega_i,h_i)=\mathrm{Vol}(M,g_m).
		\end{align*}
		\item we have
		\begin{align*}
			\lim_{i\to\infty}\Big|\Big| \big|\text{Bar}|_{N_i}^*(g_m)-h_i\big|_{h_i}\Big|\Big|_{L^\infty(\Omega_i)}=0.
		\end{align*}
	\end{enumerate}

	\textbf{Item (a):} By \eqref{eqn:property of omegas} and \eqref{eqn:pushforward of barycenter map}, it follows that $\text{Bar}_{\sharp}(D_i)$ converges in the flat topology to $[\![1_M]\!]$ inside $(M,g_m)$.
	The length estimate \eqref{eqn:length_est} guarantees that a Lipschitz bound holds uniformly in a neighborhood of $\Omega_i$ as the following: for any $\epsilon>0$, there exists $r_\epsilon>0$ such that if $i$ is sufficiently large, then for $f\in\Omega_i$ and $f'\in \mathbb{S}^+/\lambda_\Gamma(\Gamma)$ joined to $f$ by a piecewise geodesic curve $\alpha\subset\mathbb{S}^+/\lambda_\Gamma(\Gamma)$ of length less than $r_\epsilon$, we have
	\begin{equation}\label{eqn:uniform lipschitz}
		\mathrm{length}_{g_m}(\text{Bar}(\alpha))\le (1+\epsilon)\mathrm{length}(\alpha).
	\end{equation}
	As a result, we observe the following: for any $\tilde{r}\in (0,1)$, the restriction of $\text{Bar}$ to $\Omega_{i,\tilde{r}}$ is $\lambda$-Lipschitz for some $\lambda>0$ independent of $i$. In particular, this implies that $\text{Bar}|_{D_i}$ is $\lambda$-Lipschitz for sufficiently large $i$. 

	\textbf{Item (b):} It follows directly from \eqref{eqn:property of omegas}.

	\textbf{Item (c):} It follows from the fact that the Jacobian of $\text{Bar}$ is uniformly close to $1$ in $\Omega_i$ and the estimates \eqref{eqn:differential_est} and \eqref{eqn:length_est}.

	Then by using \cite[Proposition 1.4]{song2025entropy}, there is a limit map $\text{Bar}_\infty:\text{spt}(S_\infty)\to M$ which is an isometry for the intrinsic metrics induced on $\text{spt}(S_\infty)$ and $M$ by $d$ and $g_m$, respectively. Moreover, we have
	\begin{align*}
		(\text{Bar}_\infty)_\sharp(S_\infty)=[\![1_M]\!].
	\end{align*}
	Therefore, we conclude that $C_\infty$ is intrinsically isomorphic to $(M,g_m)$.
	
\end{proof}

\section{Stability for products of negatively curved symmetric spaces}\label{sec:stability of the product}
In this section, we will prove the stability for the product of negatively curved symmetric spaces. The main idea is the following: by using the intrinsic isomorphism established in Theorem \ref{thm:uniqueness}, we obtain the two sequences of subsets $\Omega_i\subset A_i\subset M$ such that the sequence of minimizing Riemannian metrics $g_i$ on $\Omega_i$ becomes close to the model metric $g_m$, and $(A_i,g_i|_{A_i})$ admits the limit map from $(M,g_m)$ into its limit space, which is $1$-Lipschitz and bi-Lipschitz. By adding appropriate `short-cuts' to $\Omega_i$ in $A_i$, we construct another sequence of subsets and show the convergence to $(M,g_m)$ in the measured Gromov-Hausdorff topology.

\begin{remark}\label{remark:notation}
	Throughout this section, for an open subset $U\subset M$ and a Riemannian metric $g$ on $M$, the notation $d_{g|_{U}}$ denotes the intrinsic distance on $U$ induced by $g$, i.e.,
	\[
		d_{g|_{U}}(x,y)=\inf_{\gamma}\mathrm{Length}(\gamma,g)
	\]
	where the infimum is taken over all rectifiable curves $\gamma$ connecting $x$ and $y$ that are contained in $U$. This should not be confused with the restriction $(d_g)|_U$ of the ambient distance function $d_g$ to $U$.
\end{remark}

\begin{theorem}\label{thm:stability of the product}
	Let $(M,g_0)$ be a closed oriented $n$-manifold defined in Theorem \ref{thm:spherical volume}. Suppose that $\{g_i\}$ is a sequence of Riemannian metrics on $M$ with $\mathrm{Vol}(M,g_i)=\mathrm{Vol}(M,g_m)$ satisfying
	\begin{equation*}
		\lim_{i\to\infty}h(g_i)=h(g_m),
	\end{equation*}
	where $g_m=\frac{h(g_{\mathrm{min}})^2}{4n}g_{\mathrm{min}}$ is the normalized metric of $g_{\mathrm{min}}$ defined in Theorem \ref{thm:spherical volume}. Then there exists a sequence of smooth subsets $Z_i\subset M$ with
	\begin{equation*}
		\lim_{i\to\infty}\mathrm{Vol}(Z_i,g_i)=0
	\end{equation*} such that $(M\setminus Z_i,d_{g_i|_{M\setminus Z_i}},d\mathrm{vol}_{g_i})$ converges to $(M,d_{g_m},d\mathrm{vol}_{g_m})$ in the measured Gromov-Hausdorff topology.
\end{theorem}

To prove the theorem, we state the following proposition which extends \cite[Theorem 3.2]{song2025entropy} to the product case. 

\begin{proposition}\label{proposition:stability of the product}
	Let $(M,g_0)$ be a closed oriented $n$-manifold defined in Theorem \ref{thm:spherical volume}. Suppose that $\{g_i\}$ is a sequence of Riemannian metrics on $M$ with $\mathrm{Vol}(M,g_i)=\mathrm{Vol}(M,g_m)$ satisfying
	\begin{equation*}
		\lim_{i\to\infty}h(g_i)=h(g_m),
	\end{equation*}
	where $g_m=\frac{h(g_{\mathrm{min}})^2}{4n}g_{\mathrm{min}}$ is the normalized metric of $g_{\mathrm{min}}$ defined in Theorem \ref{thm:spherical volume}. Then there are smooth open subsets $\Omega_i\subset A_i\subset M$ and a sequence of maps $\varphi_i:(M,g_i)\to (M,g_m)$ such that the following holds after taking a subsequence:
	\begin{enumerate}[label=(\alph*)]
		\item $\displaystyle\lim_{i\to\infty}\mathrm{Vol}(\Omega_i,g_i)=\lim_{i\to\infty}\mathrm{Vol}(A_i,g_i)=\mathrm{Vol}(M,g_m)$,
		\item $\displaystyle\lim_{i\to\infty}\Big|\Big|\big|\varphi_i^* g_m - g_i\big|_{g_i}\Big|\Big|_{L^\infty(\Omega_i)}=0$,
		\item $(A_i,g_i|_{A_i})$ converges in the intrinsic flat topology to an integral current space 
		\begin{align*}
			C_\infty=(X_\infty,d_\infty,S_\infty),
		\end{align*}
		\item there is a bi-Lipschitz bijection $\Psi:(M,g_m)\to (\mathrm{spt}\,S_\infty,d_\infty)$ which is $1$-Lipschitz and $\Psi_\sharp([\![1_M]\!])=S_\infty$. (In fact, $\Psi=\varphi_\infty^{-1}$ where $\varphi_\infty$ is the limit map of $\varphi_i$.)
		\item Letting $f_i:=\Psi\circ\varphi_i$, we have the following property: for any $\varepsilon>0$, for all $i$ sufficiently large, the map $f_i:M\to \mathrm{spt}\,S_\infty$ is a homotopy equivalence such that the restriction $f_i|_{A_i}$ is an $\varepsilon$-isometry. This implies that $(A_i,g_i|_{A_i})$ converges to $(\mathrm{spt}\,S_\infty,d_\infty)$ in the Gromov-Hausdorff sense.
	\end{enumerate}
	\begin{figure}[ht!]
		\centering
		\begin{tikzcd}
			(A_i,d_{g_i|_{A_i}}) \arrow[r, "\varphi_i"] \arrow[rd, "f_i:\varepsilon\text{-isometry}"'] & (M,d_{g_m}) \arrow[d, "\Psi:1\text{-Lipschitz}"] \\
			& (\mathrm{spt}\,S_\infty,d_\infty)
		\end{tikzcd}
		\caption{The maps $\varphi_i$, $f_i$, and $\Psi$}
	\end{figure}
\end{proposition}

\begin{remark}
	One may interpret the regions $\Omega_i$ and $A_i$ in Proposition \ref{proposition:stability of the product} as the following:
	\begin{itemize}
		\item The Jacobian and distance estimates of the barycenter map allow us to find a sequence of regions $\Omega_i$ where the Riemannian metric $g_i$ is close to $g_m$ and the volume of $\Omega_i$ converges to the volume of the whole manifold. However, the length metric induced by $g_i|_{\Omega_i}$ may be strictly larger than the one induced by $g_m$ on $M$ at some points since the choice of paths in $\Omega_i$ can be too restrictive.
		\item The region $A_i$ is constructed so that for any two points $p,q$ in $A_i$, a path in $(A_i,g_i|_{A_i})$ exists whose length is no more than $d_{g_m}(\varphi_i(p),\varphi_i(q))$ plus some small error $\epsilon$ with $\epsilon\to 0$ as $i\to\infty$. This procedure can be done by enlarging the region $\Omega_i$ carefully. Moreover, utilizing the compactness result in the intrinsic flat topology, the issue of splines in the limit space $(\mathrm{spt}\,S_\infty,d_\infty)$ can be dealt with. However, the (inverse) limit map $\Psi:(M,g_m)\to (\mathrm{spt}\,S_\infty,d_\infty)$ is not necessarily an isometry since $\Psi^* d_\infty$ may still be strictly smaller than $d_{g_m}$ at some points. 
	\end{itemize}
\end{remark}
\begin{remark}\label{remark:equidistribution}
	In \cite[Theorem 3.5]{song2025entropy}, it is shown that the map $\Psi$ is actually an isometry, hence, the desired stability is established by setting the removed subsets $Z_i:=M_i\setminus A_i$. The main idea is as follows: suppose that there is a pair of points $p,q\in M$ whose distance is strictly larger than the distance between $\Psi(p)$ and $\Psi(q)$ in the limit space. Then from the equidistribution property of geodesic spheres in negatively curved symmetric spaces, one can prove that the volume of geodesic spheres in the limit space grows significantly faster, so that the volume entropy $h(g_i)$ cannot converge to the minimal entropy; this leads to a contradiction.
	
	On the other hand, the equidistribution property of geodesic spheres does not hold in the product case. Therefore, we construct a different sequence of subsets by adding `short-cuts' to $\Omega_i$ in $A_i$ to establish the stability result. In doing so, the $(n-1)$-volume of the boundary of the subsets cannot be controlled in general, as shown by the counterexample in Section \ref{sec:counterexample}.
\end{remark}

\begin{proof}[Proof of Proposition \ref{proposition:stability of the product}]
	We only sketch the idea of the proof for completeness as the technical details can be found in the proof \cite[Theorem 3.2]{song2025entropy}.\\
	
	\textbf{Step 1. Natural maps connecting the spherical volume and the minimal entropy}\\
	
	We recall the following natural map that connects the spherical Plateau solution and the entropy minimizing sequence: let $g$ be a Riemannian metric on $M$ and $\tilde{M}$ be the universal cover of $M$. Denote a Borel fundamental domain in $\tilde{M}$ for the action of $\Gamma=\pi_1(M)$ and let $\gamma.D_M$ be its image by an element $\gamma\in\Gamma$. Define for $c>h(g)$ the map $\mathcal{P}_c:\tilde{M}\to S^\infty$ by
	\begin{align*}
		\mathcal{P}_c(x):&\,\Gamma\to\mathbb{R}\\
		&\gamma\mapsto\frac{1}{||e^{-\frac{c}{2}\text{dist}_g(x,\cdot)}||_{L^2(\tilde{M},g)}}\left(\int_{\gamma.D_M}e^{-c\, \text{dist}_g(x,u)}\,d\mu_g(u)\right)^{1/2}.
	\end{align*}
	Note that $\mathcal{P}_c$ is a $\Gamma$-equivariant Lipschitz map and for almost all $x\in\tilde{M}$, it satisfies that
	\begin{equation}\label{eqn:property of P_c}
		\sum_{j=1}^n |d_x\mathcal{P}_c(e_j)|^2\le \frac{c^2}{4},
	\end{equation}
	where $\{e_j\}$ is an orthonormal basis of $T_x\tilde{M}$ (see \cite[Lemma 3.1]{song2025entropy}).

	Utilizing those natural maps, we obtain a sequence of maps $\phi_i:(M,g_i)\to \mathbb{S}^+/\lambda_\Gamma(\Gamma)$ which is a minimizing sequence in $\mathscr{H}_M$ (achieving the spherical volume) as the following: choose a sequence of positive numbers $\{c_i\}$ with $c_i>h(g_i)$ and $\lim_{i\to\infty}c_i=h(g_m)$. Then the maps $\mathcal{P}_{c_i}:\tilde{M}\to S^\infty$ are well-defined and $\Gamma$-equivariant, so we also call the quotient map $\mathcal{P}_{c_i}:(M,g_i)\to S^\infty/\lambda_\Gamma(\Gamma)$. After a small perturbation, we obtain a sequence of smooth immersions $\phi_i:(M,g_i)\to\mathbb{S}^+/\lambda_\Gamma(\Gamma)$ which belongs to $\mathscr{H}_M$. It follows from \eqref{eqn:property of P_c} that
	\begin{equation}\label{eqn:jacobian of phi_i}
		|\text{Jac }\phi_i|\le \left(1+\frac{\nu_i}{n}\right)^{n/2}
	\end{equation}
	where the Jacobian is computed with respect to $g_i$, thus, $\mathbf{M}((\phi_i)_\sharp([\![1_M]\!]))$ converges to $\text{SphereVol}(M)=\mathrm{Vol}(M,g_m)$, i.e., $\phi_i\in \mathscr{H}_M$ is a minimizing sequence in the spherical Plateau problem.\\
	
	\textbf{Step 2. Finding a sequence of ``good'' open subsets $\Omega_i$} \\
	
	By the area formula and \eqref{eqn:jacobian of phi_i}, there are open subsets $\hat{\Omega}_i\subset M$ such that $|\text{Jac }\phi_i|\to 1$ on $\hat{\Omega}_i$ with
	\begin{equation}
		\lim_{i\to\infty}\mathrm{Vol}(\hat{\Omega}_i,g_i)=\mathrm{Vol}(M,g_m).
	\end{equation}
	Thus, by \eqref{eqn:property of P_c}, we have
	\begin{equation}\label{eqn:almost isometry of phi_i}
		\lim_{i\to\infty}\Big|\Big|\big|\phi_i^* \mathbf{g}_{\text{Hil}} - g_i\big|_{g_i}\Big|\Big|_{L^\infty(\hat{\Omega}_i)}=0
	\end{equation}
	where $\mathbf{g}_{\text{Hil}}$ is the Hilbert metric on $\mathbb{S}^+/\lambda_\Gamma(\Gamma)$. Now define $\varphi_i:=\text{Bar}\circ\phi_i$. Using \eqref{eqn:almost isometry of phi_i} and the Jacobian bound \eqref{eqn:Jac_est}, we can find open subsets $\Omega_i\subset M$ with
	\begin{equation}\label{eqn:volume of Omega_i}
		\lim_{i\to\infty}\mathrm{Vol}(\Omega_i,g_i)=\mathrm{Vol}(M,g_m),
	\end{equation}
	and
	\begin{equation}\label{eqn:almost isometry of varphi_i}
		\lim_{i\to\infty}\Big|\Big|\big|\varphi_i^* g_m - g_i\big|_{g_i}\Big|\Big|_{L^\infty(\Omega_i)}=0.
	\end{equation}
	As in the proof of Theorem \ref{thm:uniqueness}, we define the smoothed-out $r^{(i)}$-neighborhoods of $\Omega_i$ in $(M,g_i)$, say $\Omega_{i,r^{(i)}}$, satisfying that 
	\begin{itemize}
		\item the closure of $\Omega_{i,r^{(i)}}$ is a compact manifold with a smooth boundary,
		\item $\lim_{i\to\infty}\mathrm{Area}(\partial\Omega_{i,r^{(i)}},g_i)=0$, and
		\item the restriction $\varphi_i|_{\Omega_{i,r^{(i)}}}$ is uniformly Lipschitz.
	\end{itemize}
	
	Note that, by replacing $\Omega_i$ with $\Omega_i\cap A_i$ if necessary, we may always assume $\Omega_i\subset A_i$ without affecting any of the above properties.\\
	
	\textbf{Step 3. Constructing the limit map using the compactness theorem in the intrinsic flat topology}\\
	
	Let $(N_i,h_i):=(\Omega_{i,r^{(i)}},g_i|_{\Omega_{i,r^{(i)}}})$. We will use the Wenger's compactness theorem to obtain the limit map. Define, for each $i$,
	\begin{align*}
		\hat{d}_i:=\min\{d_{h_i},6\,\text{diam}(M,g_m)\},
	\end{align*}
	where $d_{h_i}$ is the distance function induced by $h_i$ on $N_i$. Let $D_i:=[\![1_{N_i}]\!]$. By the compactness theorem, the integral current spaces $D_i$ converge subsequentially to an integral current space
	\begin{align*}
		\hat{C}_\infty=(\hat{X}_\infty,\hat{d}_\infty,\hat{S}_\infty)
	\end{align*}
	in the intrinsic flat topology. This implies that there exist a Banach space $\hat{\mathbf{Z}}'$ and isometric embeddings
	\begin{align*}
		(N_i,\hat{d}_i)\hookrightarrow\hat{\mathbf{Z}}',\quad \text{spt }\hat{S}_\infty\hookrightarrow\hat{\mathbf{Z}}',
	\end{align*}
	such that $D_i$ converges to $\hat{S}_\infty$ in the flat topology inside $\hat{\mathbf{Z}}'$. Then by using \cite[Proposition 1.4]{song2025entropy}, there is a limit map
	\begin{align*}
		\varphi_{\infty}:(\text{spt }\hat{S}_\infty,\hat{d}_\infty)\to (M,g_m),
	\end{align*}
	which is Lipschitz, bijective and whose inverse $\hat{\Psi}:=\varphi_{\infty}^{-1}$ is $1$-Lipschitz with respect to the intrinsic metrics. Hence, $\hat{\Psi}$ is $1$-Lipschitz and bi-Lipschitz, and $\hat{\Psi}_\sharp([\![1_M]\!])=\hat{S}_\infty$. 

	In the above procedure, we remove the possibility of splines in the limit space $(\text{spt }\hat{S}_\infty,\hat{d}_\infty)$ by considering the metric $\hat{d}_i$ given an artificial upper bound. Once we get the limit space and the inverse of the limit map $\Psi$, it can be shown that the original integral current spaces $(N_i,d_{h_i},[\![1_{N_i}]\!])$ also converge subsequentially to $\hat{C}_\infty$ in the intrinsic flat topology. We omit the technical details.
	
	Still, note that $\hat{C}_\infty$ and $\hat{\Psi}$ are not the limit space and the limit map stated in the proposition as we will obtain them in the next step.\\
	
	\textbf{Step 4. Gromov-Hausdorff convergence and $\epsilon$-isometries}\\
	
	Now we construct the subsets $\{A_i\}$ stated in the proposition. Recall that by using the Banach space $\hat{\mathbf{Z}}$ and isometric embeddings, we regard $N_i$ and $\text{spt }\hat{S}_\infty$ as subsets of $\hat{\mathbf{Z}}$. Using the convergence in the flat topology, there exist finite subsets $\Sigma_i\subset N_i$ converging in the Hausdorff topology to $\text{spt }\hat{S}_\infty$ in $\hat{\mathbf{Z}}$. Then we can find a sequence of open subsets $A_i\subset M$ as the smoothed-out $t_i$-neighborhoods of $\Sigma_i$ in $(M,g_i)$ such that
	\begin{equation}\label{eqn:volume of A_i}
		\lim_{i\to\infty}\mathrm{Vol}(A_i,g_i)=\mathrm{Vol}(M,g_m),\quad \lim_{i\to\infty}\mathrm{Area}(\partial A_i,g_i)=0,
	\end{equation}
	and for any $s>0$ and any sequence of points $p_i\in A_i$, we have
	\begin{equation}\label{eqn:noncollapsing of A_i}
		\liminf_{i\to\infty}\mathrm{Vol}(B_{h_i|_{A_i}}(p_i,s),h_i)>0.
	\end{equation}
	Note that \eqref{eqn:noncollapsing of A_i} is the key additional property to ensure that the Gromov-Hausdorff convergence holds. 
	
	Now we repeat \textbf{Step 2} for $A_i$: as a result, it follows that $(A_i,g_i|_{A_i})$ converges subsequentially to an integral current space
	\begin{align*}
		C_\infty=(X_\infty,d_\infty,S_\infty),
	\end{align*}
	that is, there exist a Banach space $\mathbf{Z}$ and isometric embeddings
	\begin{align*}
		(A_i,d_{g_i|_{A_i}})\hookrightarrow\mathbf{Z},\quad \text{spt }S_\infty\hookrightarrow\mathbf{Z},
	\end{align*}
	such that $[\![1_{A_i}]\!]$ converges to $S_\infty$ in the flat topology inside $\mathbf{Z}$. Also, we have the limit map $\phi_\infty:(\text{spt }S_\infty,d_\infty)\to (M,g_m)$ of $\phi_i:(A_i,g_i|_{A_i})\to (M,g_m)$, whose inverse 
	\[
		\Psi:=\phi_\infty^{-1}:(M,g_m)\to (\text{spt }S_\infty,d_\infty)
	\] is $1$-Lipschitz and bi-Lipschitz. By using \eqref{eqn:noncollapsing of A_i}, one can show that $(A_i,g_i|_{A_i})$ converges to $(\text{spt }S_\infty,d_\infty)$ in the Hausdorff topology inside $\mathbf{Z}$.

	Finally, define $f_i:=\Psi\circ\varphi_i:M\to\text{spt }S_\infty$. It is straightforward to check that each $f_i$ is a homotopy equivalence and for any given $\varepsilon>0$, the restriction map
	\begin{align*}
		f_i|_{A_i}:(A_i,d_{g_i|_{A_i}})\to (\text{spt }S_\infty,d_\infty)
	\end{align*}
	is an $\varepsilon$-isometry for sufficiently large $i$. 

	Now we can verify all the properties stated in the proposition: (a) is obtained by \eqref{eqn:volume of Omega_i}, \eqref{eqn:almost isometry of varphi_i}, and \eqref{eqn:volume of A_i}, and (b) is obtained by \eqref{eqn:almost isometry of varphi_i}. The properties (c), (d), and (e) are obtained in \textbf{Step 4}.
\end{proof}

Now we turn to the proof of Theorem \ref{thm:stability of the product}.

\begin{proof}[Proof of Theorem \ref{thm:stability of the product}]
	We construct a sequence of finite graphs $G_i$ in $A_i$ using $\Omega_i, A_i,$ and $\varphi_i$ in Proposition \ref{proposition:stability of the product} such that $(G_i,d_{g_i|_{G_i}})$ converges to $(M,d_{g_m})$ in the Gromov-Hausdorff sense. Choose $\delta_i, \varepsilon_i>0$ with $\delta_i<\min\left\{\frac{\varepsilon_i}{4},\frac{\varepsilon_i^2}{6\text{diam}_{g_m}(M)}\right\}$ and $\varepsilon_i\to 0$ as $i\to\infty$ such that the map $f_i|_{A_i}$ is an $(\delta_i/2)$-isometry for every $i$ by taking a subsequence if necessary. Choose a finite set $S_i$ in $\Omega_i$ so that the sets $S_i, \varphi(S_i)$ are $\delta_i$-nets in $(\Omega_i,d_{g_i|_{\Omega_i}})$ and $(M,d_{g_m})$, respectively. Let $N_i$ be the number of points in $S_i$. Now we construct the following graph $G_i$ in $A_i$ whose vertices are the points in $S_i$: 
	\begin{itemize}
		\item Two points $x,y$ in $S_i$ are connected by an edge $e$ contained in $A_i$ if and only if
		\begin{align*}
			d_{g_m}(\varphi_i(x),\varphi_i(y))<\varepsilon_i,
		\end{align*} where the length of this edge satisfies
		\[
			\max\{0,d_{g_m}(\varphi_i(x),\varphi_i(y))-\frac{\varepsilon_i}{N_i}\}<L_{g_i}(e)<d_{g_m}(\varphi_i(x),\varphi_i(y))+\delta_i.
		\]
	\end{itemize}
	Note that such an edge can be found due to the following: since $f_i|_{A_i}$ is an $(\delta_i/2)$-isometry and $\Psi$ is $1$-Lipschitz, we have for $x,y\in S_i$ that
	\begin{align*}
		d_{g_m}(\varphi_i(x),\varphi_i(y))&\ge d_{\infty}(\Psi\circ\varphi_i(x),\Psi\circ\varphi_i(y))\\
		&\ge d_{g_i|_{A_i}}(x,y)-\frac{\delta_i}{2},
	\end{align*}
	which implies that there exists a path in $(A_i,g_i|_{A_i})$ joining $x$ and $y$ whose length is no more than $d_{g_m}(\varphi_i(x),\varphi_i(y))+\delta_i$. The lower bound of the edge length can easily be achieved. We may also choose so that any two edges in $G_i$ are disjoint.

	Now we verify that $(G_i,d_{g_i|_{G_i}})$ is an $(\varepsilon_i+\delta_i,\varepsilon_i)$-approximation of $(M,g_m)$. By definition, $S_i$ and $\varphi_i(S_i)$ are $(\varepsilon_i+\delta_i)$-net in $(G_i,d_{g_i|_{G_i}})$ and $(M,g_m)$, respectively.
	
	\textbf{Step 1.} We first show $d_{g_i|_{G_i}}(x,y)\le d_{g_m}(\varphi_i(x),\varphi_i(y))+\varepsilon_i$ for any $x,y\in S_i$. Fix any $x,y$ in $S_i$. Let $\gamma$ be a shortest path in $(M,g_m)$ connecting $x$ and $y$. Choose $\ell$ points $z_1,\ldots,z_\ell$, where $\ell+1\le 2L_{g_m}(\gamma)/\varepsilon_i$, dividing $\gamma$ into intervals of lengths no greater than $\varepsilon_i/2$. For every $j=1,\ldots, \ell$, there is a point $\varphi_i(x_j)\in S_i$ such that $d_{g_m}(\varphi_i(x_j),z_j)<\delta_i$. Set $x_0=x$, $x_{\ell+1}=y$, $z_0=\varphi_i(x)$, and $z_{\ell+1}=\varphi_i(y)$. Note that
	\[
		d_{g_m}(\varphi_i(x_j),\varphi_i(x_{j+1}))\le d_{g_m}(z_j,z_{j+1})+2\delta_i<\varepsilon_i
	\]
	for all $j=0,\ldots,\ell$ since $\delta_i<\varepsilon_i/4$, so $x_j$ and $x_{j+1}$ are connected by an edge in $G_i$. Therefore, we have
	\begin{align*}
		d_{g_i|_{G_i}}(x,y) &\le \sum_{j=0}^{\ell}d_{g_i|_{G_i}}(x_j,x_{j+1}) \\
		&\le \sum_{j=0}^{\ell}d_{g_m}(\varphi_i(x_j),\varphi_i(x_{j+1}))+\delta_i(\ell+1) \\
		&\le \sum_{j=0}^{\ell}d_{g_m}(z_j,z_{j+1})+3\delta_i(\ell+1) \\
		&= d_{g_m}(\varphi_i(x),\varphi_i(y))+3\delta_i(\ell+1) \\
		&\le d_{g_m}(\varphi_i(x),\varphi_i(y))+\delta_i\cdot \frac{6\mathrm{diam}_{g_m}(M)}{\varepsilon_i} \\
		&< d_{g_m}(\varphi_i(x),\varphi_i(y))+\varepsilon_i,
	\end{align*}

	\textbf{Step 2.} Now we want to show $d_{g_i|_{G_i}}(x,y)\ge d_{g_m}(\varphi_i(x),\varphi_i(y))-\varepsilon_i$ for any $x,y\in S_i$. Let $x_0=x, x_1,\ldots x_{\ell-1},x_\ell=y$ be the points in $S_i$ which form a shortest path connecting $x$ and $y$. Then we have
	\begin{align*}
		d_{g_i|_{G_i}}(x,y) &\ge \sum_{j=0}^{\ell-1}d_{g_i|_{G_i}}(x_j,x_{j+1}) \\
		&\ge \sum_{j=0}^{\ell-1}d_{g_m}(\varphi_i(x_j),\varphi_i(x_{j+1}))-\frac{\varepsilon_i}{N_i}(\ell-1) \\
		&\ge d_{g_m}(\varphi_i(x),\varphi_i(y))-\varepsilon_i.
	\end{align*}

	As a result, we have shown that $(G_i,g_i|_{G_i})$ converges to $(M,g_m)$ in the Gromov-Hausdorff sense. This concludes the proof by defining $\tilde{\Omega}_i:=\Omega_i\cup \tilde{G}_i$ where $\tilde{G}_i$ is a smooth open neighborhood of $G_i$ in $A_i$ so that $(\tilde{\Omega}_i,g_i|_{\tilde{\Omega}_i})$ converges to $(M,g_m)$ in the Gromov-Hausdorff sense. Note also that for $Z_i:=M\setminus\tilde{\Omega}_i$, we have
	\[
		\lim_{i\to\infty}\mathrm{Vol}(Z_i,g_i)=0.
	\]
	Let $\mu_i$ be the measure on $\tilde{\Omega}_i$ determined by the Riemannian metric $g_i|_{\tilde{\Omega}_i}$ and let $\mu$ be the measure on $M$ by $g_m$. To show the convergence in the sense of the measured Gromov-Hausdorff topology, it suffices to show that $(\varphi_i)_{\#}(\mu_i)$ converges to $\mu$ weakly. Recall the properties that
    \begin{equation*}
        ||\varphi_i^* g_m-g_i||_{L^\infty(\Omega_i)}\to 0\text{ as }i\to \infty,
    \end{equation*} and $\varphi_i$ is injective on $\Omega_i$ for large $i$. By using this fact, it is straightforward that $(\varphi_i)_{\#}(\mu_i)$ weakly converges to $\mu$. Indeed, for any open set $U$ in $M$, we have
    \begin{align*}
        \mu_i(\varphi_i^{-1}(U)) &= \mu_i(\varphi_i^{-1}(U)\cap\Omega_i)+\mu_i(\varphi_i^{-1}(U)\setminus\Omega_i)\\
        &= \mu(U\cap\varphi_i(\Omega_i))+o(1)\\
        &=\mu(U)+o(1)\text{ as }i\to \infty.
    \end{align*}
    The second line follows from the properties recalled above and that $\mu_i(\tilde{\Omega}_i\setminus\Omega_i)\to 0$, and the third line follows from the fact that $\mu(\varphi_i(\Omega_i))\to \mu(M)$.
\end{proof}

\section{A counterexample of the volume entropy stability conjecture}\label{sec:counterexample}
In this section, we construct a counterexample of the volume entropy stability statement analogous to \cite[Theorem 0.1]{song2025entropy}. For $n\ge 3$, let $\tilde{M}=\mathbb{H}^n\times\mathbb{H}^n$ equipped with the product metric $g_0=h_1\oplus h_2$. Let $M^{2n}=\tilde{M}/\Gamma_1\times\Gamma_2$ where $\Gamma_1$ and $\Gamma_2$ are cocompact lattices of $\mathrm{Isom}(\mathbb{H}^n)$.

We slightly abuse the notation that $g_0=h_1\oplus h_2$ stands for both the product metric on $\tilde{M}$ and the quotient metric on $M$. 

We choose a totally geodesic hypersurface $\tilde{S}$ in $\tilde{M}$ such that
\begin{itemize}
	\item $\tilde{S}=U\times C$ where $U$ is an open set and $C$ is a totally geodesic hypersurface in $\mathbb{H}^n$.
	\item $\tilde{S}$ is $(2n-1)$-volume-minimizing in $\tilde{M}$, i.e., any compactly supported variation of $\tilde{S}$ increases the $(2n-1)$-volume of $\tilde{S}$.
\end{itemize}
Let $S$ be the image of $\tilde{S}$ in $M$ under the quotient map. We slightly abuse the notation that $\tilde{S}$ stands for all the lifts of $S$ in $\tilde{M}$.

We define a sequence of smooth Riemannian metrics $\{g_i\}$ on $M$ as the following: for $0<\eta<1, \delta_i,\tau_i>0$ with $\delta_i,\tau_i\to 0$ as $i\to\infty$,
\begin{align}\label{eqn:counterexample}
	g_i=\left\{
		\begin{array}{ll}
			(\eta^2 h_1)\oplus h_2 & \text{ in }O_{\delta_i/2}(S),\\
			g_0 & \text{ in }M\setminus O_{\delta_i}(S),
		\end{array}
	\right.
\end{align} satisfying
\begin{align*}
	g_i(O_{\delta_i}(S))\le g_0(O_{\delta_i}(S)),\text{ and }\mathrm{Vol}_{g_i}(O_{\delta_i}(S))\le\mathrm{Vol}_{g_0}(O_{\delta_i}(S))(1+\tau_i),
\end{align*}
where $O_{\delta_i}(S)$ is the $\delta_i$-neighborhood of $S$ with respect to $g_0$.

One can show that $(M,g_i)$ converges to the integral current space $(M,d_\eta,[\![1_M]\!])$ in both the Gromov-Hausdorff and intrinsic flat senses, where the metric $d_\eta$ is defined as 
\[
	d_\eta(x_1,x_2):=\min\{\mathrm{dist}_{g_0}(x_1,x_2),\min\{\mathrm{dist}_{g_0}(x_1,y_1)+\mathrm{dist}_{(\eta^2 h_1)\oplus h_2|_{S}}(y_1,y_2)+\mathrm{dist}_{g_0}(x_2,y_2)\,:\, y_i\in S\}\}
\]
for any $x_1,x_2\in M$, and the current $[\![1_M]\!]$ is the one induced from the volume form of $g_0$. (cf. \cite[Lemmas 4.2,4.3]{Basilio:2018aa}). In the remaining part of this section, we show that this sequence of metrics serves as a counterexample.

\begin{theorem}
	Let $\{g_i\}$ be the sequence of Riemannian metrics defined in \eqref{eqn:counterexample} for $0<\eta<1$ sufficiently close to $1$. Then the following hold:
	\begin{enumerate}
		\item The volume entropy of $(M,g_i)$ converges to the minimum volume entropy of $M$ as $i\to\infty$.
		\item There does not exist a sequence of smooth subsets $Z_i\subset M$ with
		\begin{equation*}
			\lim_{i\to\infty}\mathrm{Vol}(Z_i,g_i)=\lim_{i\to\infty}\mathrm{Area}(\partial Z_i,g_i)=0
		\end{equation*} such that $(M\setminus Z_i,g_i)$ converges to $(M,g_0)$ in the measured Gromov-Hausdorff sense.
	\end{enumerate}
\end{theorem}

Note that by simply normalizing the volume, one can obtain a counterexample stated in Theorem \ref{thm:intro_counterexample} from the introduction.
 
\subsection{The volume entropy converges to the minimum volume entropy}
Here, we prove the following proposition:
\begin{proposition}\label{proposition:entropy of the counterexample}
	For $0<\eta<1$ sufficiently close to $1$, let $\{g_i\}$ be the sequence of Riemannian metrics defined in \eqref{eqn:counterexample}. Then we have
	\[
		\lim_{i\to\infty}h(g_i)=h(g_0).
	\]
\end{proposition}

Fix the origin $o=(o_1,o_2)\in\tilde{M}$, where $o_1\in\mathbb{H}^n$ is arbitrary and $o_2\in\mathbb{H}^n$ is chosen so that it does not lie in any lift of \(C\). This is
possible since the lifts of \(C\) form a countable collection of totally
geodesic hypersurfaces. Let $(r_i,\omega_i), i=1,2,$ be the geodesic spherical coordinates centered at $o_i$ on each factor of $\tilde{M}$. For any given point $x$ in $\tilde{M}$, let 
\begin{equation}\label{eqn:angle between radials}
	\theta(x)=\tan^{-1}(r_1(x)/r_2(x)).
\end{equation} 
To prove Proposition~\ref{proposition:entropy of the counterexample}, we first establish the following elementary lemma. From the construction of the metric $d_\eta$, any length-minimizing path joining two points in $\tilde M$ with respect to $d_\eta$ consists of geodesic segments in $(\tilde M, g_0)$ and segments contained
in a lift of $\tilde S$. The lemma below shows that $d_\eta(o,\cdot)=d_{g_0}(o,\cdot)$ on a conical region $R_c$ around the diagonal $\{\theta=\pi/4\}$. Since the volume growth of geodesic balls in the product metric $g_0$ concentrates near this diagonal, this ensures that the volume growth of the metric $d_\eta$ is close to that of $g_0$.

\begin{lemma}\label{lemma:conical region}
For $\eta<1$ sufficiently close to $1$, there exists $c>0$ such that
	\begin{equation}\label{eqn:conical region}
		\mathrm{dist}_{g_0}(o,x)=d_{\eta}(o,x)\text{ for all }x\text{ in } R_c,
	\end{equation}
	where the region $R_c$ is defined by
	\[ 
		R_c:=\{x\in\tilde{M}\,:\, |\theta(x)-\pi/4|\le c\}.
	\]
\end{lemma}

\begin{proof}[Proof of Lemma \ref{lemma:conical region}]
	\begin{figure}[ht]
		\begin{tikzpicture}
		
		\draw (0,0) -- (2,2) -- (8,2) -- (6,0) -- cycle;
		
		\node at (5.5,-0.3) {$\tilde{S}$};
		
		\draw (1,2.5) -- (3.5,0.5);
		\node at (1,2.5) [above] {$o$}; 

		\draw (3.5,0.5) -- (5.5,1);
		\node at (3.5,0.5) [below] {$q$}; 

		\draw (5.5,1) -- (7.2,-0.1);
		\node at (5.5,1) [above] {$r$}; 

		\draw[dashed] (3,0.3) .. controls (5,1) and (3.5,1.7) .. (3.5,1.7);
		
		\draw[thick, blue] (1,2.5) -- (4,1) -- (5.5,1);
		\node at (4,1) [above right] {$q'$}; 
		
		\end{tikzpicture}
		\caption{Shorter path when the angular coordinate is not fixed}
		\label{fig:shorter-path}
	\end{figure}
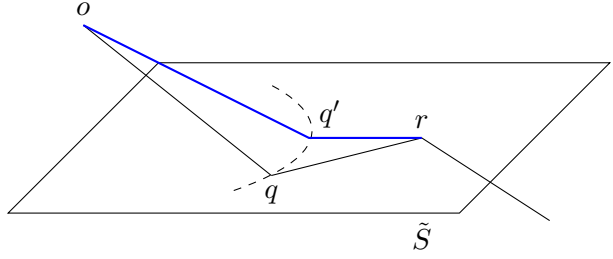

	Fix any $x\ne o$ in $\tilde{M}$. Let $\gamma$ be the length-minimizing path joining $o$ to $x$ with respect to $d_{\eta}$. It suffices to show that no $d_\eta$-path from $o$ to $x$ is shorter than the radial $g_0$-geodesic when $x\in R_c$ for some $c>0$.
	
	Suppose that $\gamma$ has corners. One can deduce that the length-minimizing path joining $o$ to any other point $x$ must have the angular coordinates fixed by $\omega_1(x)$ and $\omega_2(x)$: suppose it is not the case. Before the first corner point \(q\), the path \(\gamma\) is a \(g_0\)-geodesic segment from \(o\) to \(q\), and hence it is radial in each factor. If the following segment in $\tilde{S}$ changes the angular coordinates, then we can find a shorter path by turning the segment in $\tilde{M}\setminus S$ while fixing the escaping point $r$. See Figure \ref{fig:shorter-path} for an illustration; the length of the path connecting $o$ and $q$ and that of the one connecting $o$ and $q'$ are equal, but the path connecting $o$, $q'$, and $r$ is shorter than the one connecting $o$, $q$, and $r$. This contradicts the fact that $\gamma$ is length-minimizing. Using this argument inductively at each corner point, we can verify that the angular coordinates along $\gamma$ must be fixed.
	
	We now prove the desired distance comparison by a direct minimization argument. Let $x\in \tilde M$, and write
	\[
    	r_1=r_1(x),\qquad r_2=r_2(x),\qquad
    	r=\operatorname{dist}_{g_0}(o,x)=\sqrt{r_1^2+r_2^2}.
	\]
	By the preceding argument we may assume that a $d_\eta$-length minimizing path $\gamma$ from $o$ to $x$ keeps the angular coordinates $\omega_1,\omega_2$ fixed; the problem therefore reduces to the $(r_1,r_2)$-plane.

	Let $u$ denote the total signed $r_1$-displacement of the portions of $\gamma$ lying in $\tilde S$, and let $u_1,\ldots,u_m$ denote the signed $r_1$-displacements of the individual $\tilde S$-portions, so that $u=\sum_j u_j$. Since $g_i=(\eta^2 h_1)\oplus h_2$ in a neighborhood of $S$ and $\tilde S=U\times C$ with $U$ open in the first factor, the shortening by the factor $\eta$ applies to the first-factor component of the length along $\tilde S$. With angular coordinates fixed, the second-factor component of each $\tilde{S}$-portion lies in a single lift of $C$ and on the geodesic ray $\{\omega_2=\omega_2(x)\}$ emanating from $o_2$. Since $o_2$ lies off every lift of $C$, this ray is contained in no lift of $C$, and a geodesic ray in $\mathbb{H}^n$ not contained in a totally geodesic hypersurface meets it in at most one point. Hence the second-factor component is constant along the $\tilde{S}$-portion, and the portion has zero $r_2$-displacement. If its signed \(r_1\)-displacement is \(u_j\), then, since the metric is shortened by the factor \(\eta\) in the first factor along \(\tilde S\), this portion contributes exactly \(\eta |u_j|\) to the \(d_\eta\)-length. Hence the total contribution of the \(\tilde S\)-portions is
	\[
        \eta\sum_j |u_j|\ge \eta |u|.
	\]
 
	The remaining $g_0$-geodesic portions of $\gamma$ have total displacement $(r_1-u, r_2)$ in the $(r_1,r_2)$-plane. Since the angular coordinates are fixed, each such portion is a Euclidean segment in this plane, so by the triangle inequality in $\mathbb{R}^2$ their total $g_0$-length is at least $\sqrt{(r_1-u)^2+r_2^2}$. Consequently
	\begin{equation}\label{eqn:F lower bound}
    	L_\eta(\gamma)\ge F(u):=\sqrt{(r_1-u)^2+r_2^2}+\eta|u|.
	\end{equation}
 
	We now minimize $F$ over $u\in\mathbb{R}$.
 
	\emph{The region $u<0$.} Here $r_1-u>r_1\ge 0$, and
	\[
    	F'(u)=-\frac{r_1-u}{\sqrt{(r_1-u)^2+r_2^2}}-\eta<0,
	\]
	so $F$ is strictly decreasing on $(-\infty,0)$. The minimum of $F$ on $(-\infty,0]$ is therefore attained at $u=0$.
 
	\emph{The region $u>0$.} A critical point of $F$ in $(0,\infty)$ satisfies
	\[
    	\frac{r_1-u}{\sqrt{(r_1-u)^2+r_2^2}}=\eta.
	\]
	For $u\in (0,\infty)$, the left-hand side takes values in $(-1,\sin\theta(x))$, with supremum $\sin\theta(x)=r_1/r$ approached as $u\to 0^+$. Since $\eta>0$, the critical-point equation has a solution in $(0,\infty)$ if and only if $\eta<\sin\theta(x)$. In particular, if $\theta(x)\le\arcsin\eta$, there is no critical point in $(0,\infty)$. Since $F'(0^+)=\eta-\sin\theta(x)\ge 0$ and $F(u)\to\infty$ as $u\to\infty$, the continuity of $F'$ on $(0,\infty)$ together with the absence of zeros there forces $F'\ge 0$ on $(0,\infty)$, i.e., $F$ is non-decreasing on $[0,\infty)$.
 
	\emph{Conclusion.} Combining both regions, when $\theta(x)\le\arcsin\eta$ the minimum of $F$ on $\mathbb{R}$ is attained at $u=0$ with value
	\[
    	F(0)=\sqrt{r_1^2+r_2^2}=\operatorname{dist}_{g_0}(o,x).
	\]
	By \eqref{eqn:F lower bound}, every $d_\eta$-path from $o$ to $x$ has length at least $\operatorname{dist}_{g_0}(o,x)$. Since the radial $g_0$-geodesic from $o$ to $x$ is itself an admissible $d_\eta$-path with length exactly $\operatorname{dist}_{g_0}(o,x)$, we conclude
	\[
    	d_\eta(o,x)=\operatorname{dist}_{g_0}(o,x)
    	\qquad\text{whenever }\theta(x)\le\arcsin\eta.
	\]
 
	Finally, choose $\eta<1$ sufficiently close to $1$, and then choose $c>0$ so small that
	\[
    	\frac{\pi}{4}+c\le \arcsin\eta.
	\]
	For every $x\in R_c$ we have $\theta(x)\in[\pi/4-c,\pi/4+c]$ and hence $\theta(x)\le\arcsin\eta$. Applying the previous conclusion yields
	\[
    	d_\eta(o,x)=\operatorname{dist}_{g_0}(o,x)
    	\qquad\text{for all }x\in R_c,
	\]
	which proves \eqref{eqn:conical region}.
\end{proof}

\begin{proof}[Proof of Proposition \ref{proposition:entropy of the counterexample}]
	Using Lemma \ref{lemma:conical region}, we can see that $\mathrm{dist}_{g_i}(o,x)$ converges to $\mathrm{dist}_{g_0}(o,x)$ for any $x$ in $R_c$ as $i\to\infty$. Moreover, the length-minimizing geodesic joining $o$ and $x$ in $(\tilde{M},g_i)$ can be divided into finitely many segments such that the projection of each segment is length-minimizing in $(M,g_i)$. By comparing the lengths of each segment with respect to $g_i$ and $g_0$, it follows that 
	\[
		\mathrm{dist}_{g_i}(o,x)\le (1+\lambda_i)\mathrm{dist}_{g_0}(o,x) \text{ for all }x\in R_c
	\] for some sequence $\lambda_i>0$ with $\lambda_i\to 0$ as $i\to\infty$. 
	
	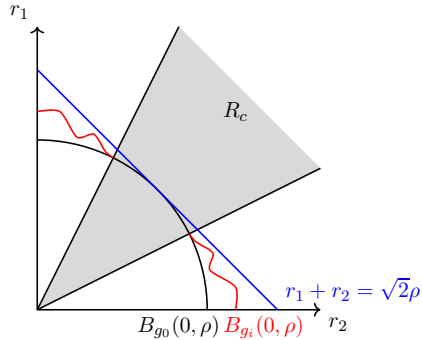
\begin{figure}[ht]
		\centering
		\resizebox{2.3in}{!}{
		\begin{tikzpicture}
			\draw[thick,->] (0,0) -- (5,0) node[anchor=north west] {$r_2$};
			\draw[thick,->] (0,0) -- (0,5) node[anchor=south east] {$r_1$};

			\fill[gray!30] (0,0) -- (2.5,5) -- (5,2.5) -- cycle;
			\draw[thick] (0,0) -- (2.5,5);
			\draw[thick] (0,0) -- (5,2.5);
			\draw[thick] (3,0) arc[start angle=0, end angle=90, radius=3cm];
			
			\draw[blue,thick] (0,4.242) -- (4.242,0) node[anchor=south west] {$r_1 + r_2 = \sqrt{2}\rho$};

			\draw[red,thick,smooth] plot coordinates {(0,3.5) (0.4,3.5) (0.7,3.05) (1,3.1) (1.2,2.8) (1.3416,2.6832)};
			\draw[red,thick,smooth] plot coordinates {(3.5,0) (3.5,0.4) (3.05,0.7) (3.1,1) (2.8,1.2) (2.6832,1.3416)};
			
			
			\node at (3.5,3.5) {$R_c$};
			\node at (2.5,-0.3) {$B_{g_0}(0,\rho)$};
			\node[text=red] at (4,-0.3) {$B_{g_i}(0,\rho)$};
		\end{tikzpicture}
		}
		\caption{The region $R_c$ and the comparison of the volume growth of $B_{g_0}(0,\rho)$ and $B_{g_i}(0,\rho)$}
		\label{fig:comparison}
	\end{figure}
	
	
	In the region $B_{g_i}(o,\rho)\cap(\tilde{M}\setminus R_{c})$, we have $r_1+r_2<\sqrt{2}\, \rho$ for sufficiently large $i$, provided that $\eta$ is sufficiently close to $1$, see Figure \ref{fig:comparison}. As a result, the volume growth of $B_{g_i}(o,\rho)$ is dominated by $\sinh^{n-1}(r_1)\sinh^{n-1}(r_2)\approx e^{(\sqrt{2}(n-1)+\lambda_i)\rho}$ for any $\rho>0$. Therefore, we have
	\[
		h(g_i)=\sqrt{2}(n-1)+o(1).
	\]
\end{proof}

\subsection{The stability statement fails}
To show the stability statement fails, suppose, on the contrary, that there exists a sequence of smooth subsets $Z_i\subset M$ with $\lim_{i\to\infty}\mathrm{Vol}(Z_i,g_i)=\lim_{i\to\infty}\mathrm{Area}(\partial Z_i,g_i)=0$ such that $(M\setminus Z_i,\mathrm{dist}_{g_i|_{M\setminus Z_i}})$ converges to $(M,\mathrm{dist}_{g_0})$ in the measured Gromov-Hausdorff topology.
%
We set the following with respect to $(M,g_0)$ (see Figure \ref{fig:parallelogram}):
\begin{itemize}
	\item Fix any $p\in S$ and $\epsilon>0$ sufficiently small such that $\exp_p(B(0,\epsilon)\cap T_p S)\subset S$ where $B(0,\epsilon)$ is the open ball in $T_p M$ centered at $0$ with radius $\epsilon$.
	\item Choose a vector $\nu_p\in T_p S$ at $p$ lying in $TU$. (Recall that the hypersurface $S$ is chosen so that its lift $\tilde{S}$ is $\tilde{S}=U\times C$ where $U$ is an open subset and $C$ is a totally geodesic hypersurface in $\mathbb{H}^n$.) Denote by $\nu$ a vector field on a neighborhood of $p$ in $S$ obtained by parallel transport of $\nu_p$.
	\item Denote by $\vec{n}_p$ the unit normal vector to $S$ at $p$.
	\item Let $H$ be the $(2n-2)$-dimensional subspace of $T_p S$ that is orthogonal to both $\nu_p$ and $\vec{n}_p$. Define $V(p,\epsilon):=\exp_p(B(0,\epsilon)\cap H)$.
	\item For each $x\in V(p,\epsilon)$, denote by $\gamma_x$ the geodesic emanating from $x$ in the direction of $\nu(x)$ whose length is equal to some positive number $\ell>0$. Here, $\ell$ is chosen so that $\gamma_x$ is contained in $S$.
	\item For each $x\in V(p,\epsilon)$, define a $2$-dimensional surface $P_x$ in $M$ as the image of the map $(t,s)\mapsto\exp_{\gamma_x(s)}(t\,\vec{n}(\gamma_x(s)))$ for $t\in(-\epsilon,\epsilon)$ and $s\in[0,\ell]$. In other words, $P_x$ is the normal $\epsilon$-neighborhood of $\gamma_x$ in the $2$-dimensional plane spanned by $\nu$ and $\vec{n}$.
\end{itemize}
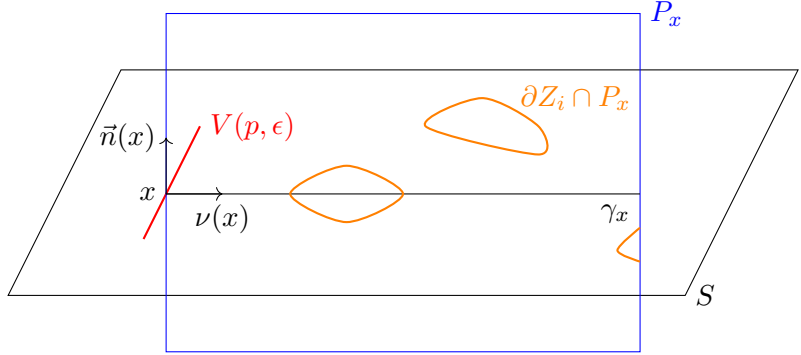
\begin{figure}[ht]
    \centering
    \begin{tikzpicture}[scale=1.5]
        \coordinate (A) at (0,0);
        \coordinate (B) at (6,0);
        \coordinate (C) at (7,2);
        \coordinate (D) at (1,2);
		\coordinate (P) at (1.2,0.5);
		\coordinate (Q) at (1.7,1.5);
		\coordinate (R) at (1.4,2.5);
		\coordinate (S) at (1.4,-0.5);
		\coordinate (T) at (5.6,2.5);
		\coordinate (U) at (5.6,-0.5);
		\coordinate (X) at (1.4,0.9);
		\coordinate (Y) at (1.9,0.9);
		\coordinate (Z) at (1.4,1.4);
        
        \draw (A) -- (B) -- (C) -- (D) -- cycle;
        \draw[red, thick] (P) -- (Q);
		\draw[blue] (R) -- (S) -- (U) -- (T) -- cycle;
		\draw[->] (X) -- (Y);
		\draw[->] (X) -- (Z);
		\draw (X) -- (5.6,0.9);

		\draw[orange, thick] plot [smooth cycle] coordinates {(2.5,0.9) (3,1.15) (3.5,0.9) (3,0.65)};
		\draw[orange, thick] plot [smooth cycle] coordinates {(3.7,1.5) (4.2,1.75) (4.7,1.5) (4.7,1.25)};
		\draw[orange, thick] plot [smooth] coordinates {(5.6,0.6) (5.4,0.4) (5.6,0.3)};

		\node at (X) [left] {$x$};
		\node[blue] at (T) [right] {$P_x$};
		\node[red] at (Q) [right] {$V(p,\epsilon)$};
		\node at (Y) [below] {$\nu(x)$};
		\node at (Z) [left] {$\vec{n}(x)$};
		\node at (B) [right] {$S$};
		\node at (5.6,0.9) [below left] {$\gamma_x$};
		\node[orange] at (4.45,1.75) [right] {$\partial Z_i\cap P_x$};

    \end{tikzpicture}
    \caption{There exists a point $x\in V(p,\epsilon)$ such that $\partial Z_i\cap P_x$ has small $\mathscr{H}^1$ measure.}
    \label{fig:parallelogram}
\end{figure}

Observe that for any $x\in V(p,\epsilon)$, there is a pair of points $p_x$ and $q_x$ in $P_{x}$ such that they lie in the image of $\{t\in(0,\epsilon)\}\times\gamma_x$ under the above parametrization of $P_x$, and the length minimizing path joining $p_x$ and $q_x$ in $(M,d_\eta)$, say $\xi_x$, contains a segment in $S$. We may choose $p_x$ and $q_x$ for each $x\in V(p,\epsilon)$ so that $d_\eta(p_x,q_x)$ remains constant, say $L$.

Moreover, the reflections $\hat{p}_x$ and $\hat{q}_x$ of $p_x$ and $q_x$, respectively, with respect to $S$ have the length minimizing path $\hat{\xi}_x$ in $(M,d_\eta)$ that contains the same segment in $S$. Such two pairs cannot exist in a smooth Riemannian manifold. We will show that the Gromov-Hausdorff limit of $(M\setminus Z_i,\mathrm{dist}_{g_i|_{M\setminus Z_i}})$ must contain such pairs of points, which leads to a contradiction.

Note that $Z_i\cap P_x$, $\partial Z_i\cap P_x$ among all $x\in V(p,\epsilon)$ decompose subsets of $Z_i$ and $\partial Z_i$ into $2$-dimensional planes and $1$-dimensional curves, respectively. By using Fubini's theorem and the condition that $\mathrm{Vol}(Z_i,g_i), \mathrm{Area}(\partial Z_i,g_i)\to 0$, we can find a sequence of points $x_i\in V(p,\epsilon)$ such that 
\[
	\mathscr{H}_{g_i}^2(Z_i\cap P_{x_i}), \mathscr{H}_{g_i}^1(\partial Z_i\cap P_{x_i})\to 0\text{ as }i\to\infty.
\]

Next, we choose paths $\sigma_i$ and $\hat{\sigma}_i$ in $M\setminus Z_i$ joining $p_i$, $q_i$ and $\hat{p}_i$, $\hat{q}_i$, respectively, close to the length minimizing paths in $(M,d_\eta)$. Whenever the paths intersect $Z_i$, we take a detour around $Z_i$ to avoid the intersection. If either of $p_i,q_i,\hat{p}_i,\hat{q}_i$ lies in $Z_i\cap P_{x_i}$, then we can choose a point on $\partial Z_i\cap P_{x_i}$ close to the point (this can be done by the fact that $\mathscr{H}_{g_i}^2(Z_i\cap P_{x_i})\to 0$), replace it with the point, and connect to the corresponding path $\sigma_i$ or $\hat{\sigma}_i$. See Figure \ref{fig:minimizing path} for an illustration. Then the following inequalities hold:
\begin{align*}
	&\mathrm{dist}_{g_i}(p_i,q_i)\le L_{g_i}(\sigma_i)\le L+o(1)+C \mathscr{H}^1(\partial Z_i\cap P_{x_i}),\\
	&\mathrm{dist}_{g_i}(\hat{p}_i,\hat{q}_i)\le L_{g_i}(\hat{\sigma}_i)\le L+o(1)+C \mathscr{H}^1(\partial Z_i\cap P_{x_i}).
\end{align*}
The lower bound is obvious by the definition of $\mathrm{dist}_{g_i}$. The $o(1)$ terms in the upper bound come from the difference from $g_i$ and $g_0$ in a small neighborhood of $S$. Using the fact that $(M,\mathrm{dist}_{g_i})\to (M,d_\eta)$ in the Gromov-Hausdorff sense, it follows that $\mathrm{dist}_{g_i}(p_i,q_i)$ converges to $L$ as $i\to\infty$. 

\begin{figure}[ht]
    \centering
    \begin{tikzpicture}[scale=2]
        \coordinate (A) at (0,0);
        \coordinate (B) at (6,0);
        \coordinate (C) at (7,2);
        \coordinate (D) at (1,2);
		\coordinate (P) at (1.2,0.5);
		\coordinate (p) at (1.7,1.3);
		\coordinate (pp) at (2,0.9);
		\coordinate (q) at (5,1.3);
		\coordinate (qq) at (4.7,0.9);
		\coordinate (p') at (1.7,0.5);
		\coordinate (q') at (5,0.5);
		\coordinate (R) at (1.2,2.3);
		\coordinate (S) at (1.2,-0.3);
		\coordinate (T) at (5.6,2.3);
		\coordinate (U) at (5.6,-0.3);
		\coordinate (X) at (1.4,0.9);
		\coordinate (Y) at (1.9,0.9);
		\coordinate (Z) at (1.4,1.4);
        
        \draw (A) -- (B) -- (C) -- (D) -- cycle;
		\draw[blue] (R) -- (S) -- (U) -- (T) -- cycle;
		\draw[thick] (p) -- (pp) -- (qq) -- (q);
		\draw[thick] (p') -- (pp) -- (qq) -- (q');

		\draw[orange, thick] plot [smooth cycle] coordinates {(2.5,0.9) (3,1.15) (3.5,0.9) (2.7,0.65)};
		\draw[orange, thick] plot [smooth cycle] coordinates {(3.7,1.5) (4.2,1.75) (4.7,1.5) (4.7,1.25)};
		\draw[orange, thick] plot [smooth] coordinates {(5.6,0.6) (5.4,0.4) (5.6,0.3)};
		\draw[orange, thick] plot [smooth cycle] coordinates {(1.3,1.3) (1.8,1.4) (1.8,1.1) (1.6,1.1)};

		\draw[red, thick] plot [smooth] coordinates {(1.83,1.4) (1.87,1.1)  (2,0.93) (2.4,0.92) (3,1.15) (3.6,0.95) (4.6,0.9) (q)};
		\draw[violet, thick] plot [smooth] coordinates {(p') (2,0.87) (2.4,0.87) (2.7,0.65) (3.7,0.88) (4.65,0.87) (q')};

		\node[blue] at (T) [right] {$P_{x_i}$};
		\node at (p) [left] {$p_i$};
		\node at (q) [right] {$q_i$};
		\node at (p') [left] {$\hat{p}_i$};
		\node at (q') [right] {$\hat{q}_i$};
		\node[red] at (3.8,0.9) [above] {$\sigma_i$};
		\node[violet] at (3.8,0.9) [below] {$\hat{\sigma}_i$};
		\node at (B) [right] {$S$};
		\node[orange] at (4.8,1.95) [below] {$\partial Z_i\cap P_{x_i}$};

    \end{tikzpicture}
    \caption{Find $\sigma_i$ and $\hat{\sigma}_i$ such that $L_{g_i}(\sigma_i)$ and $L_{g_i}(\hat{\sigma}_i)$ converge to $L$}
    \label{fig:minimizing path}
\end{figure}
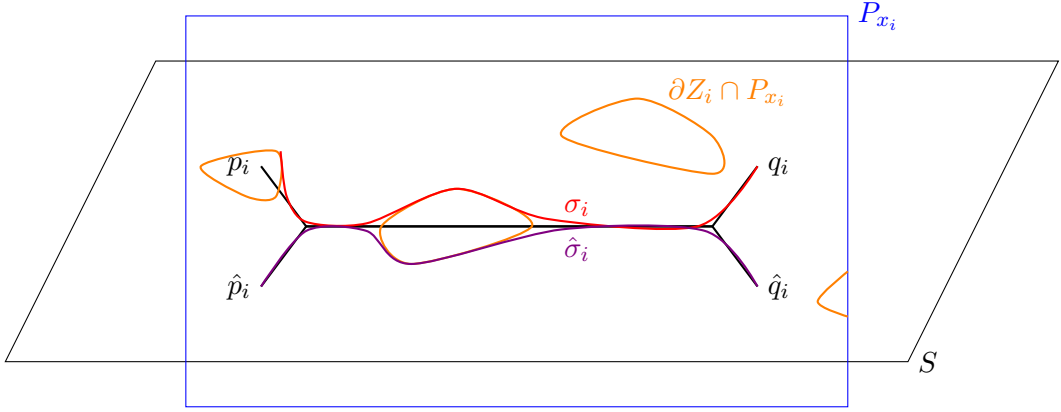

Observe that the Hausdorff distance in $(M,g_i)$ between $\sigma_i$ (or $\hat{\sigma}_i$) and the length minimizing geodesic $\zeta_i$ joining $p_i,q_i$ (or $\hat{\zeta}_i$ joining $\hat{p}_i,\hat{q}_i$) in $(M,g_i)$ converges to zero, respectively, since 
\begin{itemize}
	\item $\sigma_i$ and $\hat{\sigma}_i$ agree with the length minimizing paths $\xi_{x_i}$ and $\hat{\xi}_{x_i}$ joining $p_i, q_i$ and $\hat{p}_i,\hat{q}_i$, respectively, in $(M,d_\eta)$ for the most part as $i\to\infty$ because $\mathscr{H}_{g_i}^1(\partial Z_i\cap P_{x_i})\to 0$, and
	\item the Hausdorff distance in $(M,g_i)$ between $\zeta_i$ (or $\hat{\zeta}_i$) and $\xi_{x_i}$ (or $\hat{\xi}_{x_i})$ converges to zero as $i\to\infty$ (since $(M,g_i)$ converges to $(M,d_\eta)$ in the Gromov-Hausdorff sense as length spaces).
\end{itemize}
The above argument implies that the Gromov-Hausdorff limit of $(M\setminus Z_i,\mathrm{dist}_{g_i|_{M\setminus Z_i}})$ contains two pairs of points $p,q$ and $\hat{p},\hat{q}$ such that the length minimizing paths joining each pair share a segment. As mentioned earlier, such pairs of points cannot exist in a smooth Riemannian manifold, which completes that the stability fails.

\bibliographystyle{alpha}
\bibliography{My_Library}

\end{document}